\documentclass[12pt]{amsart}
\usepackage{mathtools,amsthm,amssymb}
\usepackage[left=1.5in,right=1.5in,top=1.5in,bottom=1.5in]{geometry}
\usepackage{graphicx}
\usepackage{xcolor}
\usepackage{hyperref}
\usepackage[mathscr]{euscript}
\usepackage{charter,eulervm}
\usepackage{verbatim}
\usepackage{setspace}
\usepackage{csquotes}
\usepackage{bbold}
\usepackage{tikz-cd}
\tikzcdset{arrows={line width=rule_thickness},arrow style=math font}
\usetikzlibrary{matrix, decorations.pathmorphing, arrows}
\usepackage{harpoon}

\theoremstyle{plain}
\newtheorem{theorem}{Theorem}[subsection]
\newtheorem*{theorem*}{Theorem}
\newtheorem{corollary}[theorem]{Corollary}
\newtheorem*{corollary*}{Corollary}
\newtheorem{proposition}[theorem]{Proposition}
\newtheorem*{proposition*}{Proposition}

\newtheorem{lemma}[theorem]{Lemma}
\newtheorem*{lemma*}{Lemma}

\theoremstyle{definition}

\newtheorem*{definition*}{Definition}

\newtheorem{example}[theorem]{Example}

\theoremstyle{remark}

\newtheorem*{remarks*}{Remarks}
\newtheorem{question}[theorem]{Question}
\newtheorem*{question*}{Question}
\newtheorem*{notation*}{Notation}

\newtheorem*{claim*}{Claim}
\newtheorem*{terminology*}{Terminology}

\newtheorem*{caution*}{Caution}

\DeclareMathOperator{\con}{con}
\DeclareMathOperator{\op}{op}
\DeclareMathOperator{\fix}{fix}

\DeclareMathOperator{\regfltr}{rflt}

\newcommand{\qtq}[1]{\quad\text{#1}\quad}
\newcommand{\setof}[2]{\left\{\, #1 : #2 \,\right\}}
\newcommand{\ssetof}[2]{\left\{#1\right\}_{#2}}
\newcommand{\map}[3]{#1 \colon #2 \to #3}

\newcommand{\mbf}[1]{\mathbf{#1}}
\newcommand{\mcal}[1]{\mathcal{#1}}
\newcommand{\mscr}[1]{\mathscr{#1}}
\newcommand{\mfrak}[1]{\mathfrak{#1}}
\newcommand{\bmfrak}[1]{\boldsymbol{\mfrak{#1}}}
\newcommand{\mbb}[1]{\mathbb{#1}}

\newcommand{\sbv}[2]{\smashoperator[#1]{\bigvee_{#2}}}
\newcommand{\sbw}[2]{\smashoperator[#1]{\bigwedge_{#2}}}
\newcommand{\sbcup}[2]{\smashoperator[#1]{\bigcup_{#2}}}
\newcommand{\sbcap}[2]{\smashoperator[#1]{\bigcap_{#2}}}
\newcommand{\downset}[1]{\left\downarrow{#1}\right\downarrow}
\newcommand{\upset}[1]{\left\uparrow{#1}\right\uparrow}

\newcommand{\combel}{\prec \!\! \prec}

\newcommand{\stst}{^{**}}
\newcommand{\st}{^*}

\newcommand{\xra}[1]{\xrightarrow{#1}}

\newcommand{\EM}{(\bmfrak{E}, \bmfrak{M})}
\newcommand{\chix}[1]{\overline{#1}}
\newcommand{\splt}[2]{\big({#2}, \chix{#1}\big)}
\newcommand{\spltt}[2]{\left({#2}, \chix{#1}\right)}

\newcommand{\R}{\mscr{R}}

\newcommand{\Topol}{\mscr{O}}

\newcommand{\W}{\mbf{W}}

\allowdisplaybreaks[1]

\begin{document}
	
\title[Pointless parts of completely regular frames]{Pointless parts of \\completely regular frames}
\author[R. N. Ball]{Richard N. Ball}
\address[Ball]{Department of Mathematics, University of Denver, Denver, Colorado 80208, U.S.A.}
\email{rball@du.edu}
\dedicatory{Dedicated to the memory of Bernhard Banaschewski, inspiration and friend.}
\date{April 28, 2023}
%\thanks{These are working notes and are not meant for distribution.}
%\thanks{File name: Pointless part.tex}
\begin{abstract}
	(Completely regular) locales generalize (Tychonoff) spaces; indeed, the passage from a locale to its spatial sublocale is a well understood coreflection. But a locale also possesses an equally important pointless sublocale, and with morphisms suitably restricted, the passage from a locale to its pointless sublocale is also a coreflection. Our main theorem is that every locale can be uniquely represented as a subdirect product of its pointless and spatial parts, again with suitably restricted projections. We then exploit this representation by showing that any locale is determined by (what may be described as) the placement of its points in its pointless part. 
\end{abstract}
\maketitle

%\tableofcontents

\section{Introduction}
The primary motivation for point free topology comes from classical point set topology. When expressed in terms of the underlying frames, the connection between these two worlds is the well known functor $\map{\sigma}{\mbf{F}}{\mbf{sF}}$, which assigns to each (completely regular) frame $L$ its spatial part $\sigma L$. (Here $\mbf{F}$ is the category of completely regular frames with frame homomorphisms, and $\mbf{sF}$ is its full subcategory of spatial frames, i.e., frames in which every element is the meet of the maximal elements above it.) The functor $\sigma$ is an epireflection, and the reflector for a given frame $L$ is the surjection 
\[
	\map{\sigma_L}{L}{\sigma L}
	\equiv \setof{b \in L}{b = \bigwedge \upset{b}_{\max L}}
	= \left(b \mapsto \bigwedge \upset{b}_{\max L} \right),
	\qquad b \in L.
\]

One of the principal advantages of the point free approach to general topology is its increase in extent beyond the spatial situation, and this article provides evidence of the benefits of that generality. We call a frame \emph{pointless} if it has no maximal elements, and we show that every frame $L$ has a pointless part $\pi L$ which plays a role roughly complementary to its spatial part. We need to restrict the homomorphisms slightly to those we term \emph{skinny}, and work with the restricted cateogory $\mbf{Fs}$ of frames with skinny morphisms and its full subcategory of pointless frames $\mbf{plFs}$. (When the domain and codomain are spatial, the skinny frame homomorphisms are those whose pointed continuous functions have scattered fibers.) In that context the functor $\map{\pi}{\mbf{Fs}}{\mbf{plFs}}$ is an epireflection, and a reflector for the frame $L$ is the surjection  
\[
	\map{\pi_L}{L}{\pi L}
	\equiv \setof{a \in L}{\forall b > a\, \exists c\ (b > c > a)}
	= \left(b \mapsto \bigwedge \upset{b}_{\pi L}\right),
	\ b \in L.
\]

Though disjoint, the two sublocales $\sigma L$ and $\pi L$ are not complementary. Nevertheless, their reflectors $\sigma_L$ and $\pi_L$ are diagnostic when taken together, in the sense that every frame $L$ has a unique representation as a subdirect product of its pointless and spatial parts, with suitably restricted projection maps. This is our main Theorem \ref{Thm:1}. The targets of this representation, here termed \emph{fat}, form a monoreflective subcategory of $\mbf{Fs}$, and even though we would like to know more about these objects (cf.\ Question \ref{Ques:1}), Theorem \ref{Thm:3} provides a fairly concrete description of the reflector arrow. 

Thus a frame $L$ is determined by the interaction of its spatial and pointless parts, and this interaction is governed by two principal connections. The first of these is the arrow $\map{\lambda_L}{\pi L}{\pi \sigma L}$ induced by applying the $\pi$ functor to the $\sigma_L$ arrow.
\[
	\begin{tikzcd}
		L \arrow{r}{\sigma_L} \arrow{d}[swap]{\pi_L}
		& \sigma L \arrow{d}{\pi_{\sigma L}}\\
		\pi L \arrow{r}[swap]{\lambda_L} 
		& \pi \sigma L
	\end{tikzcd}
\]
We refer to $\lambda_L$ as the \emph{ligature of $L$} (Subsection \ref{Subsec:Lig}). 

The second principal connection between the spatial and pointless parts of a frame is motivated by the crucial observation that each maximal element $a$ of a frame $L$ is associated with the filter $y_a \equiv \setof{b \in \pi L}{b \nleq a}$ on $\pi L$. Such filters have two key features: they are \emph{round,} i.e., for each $b \in y_a$ there exists an element $c \in y_a$ such that $c \combel b$, and $\bigvee_{y_a}b\st = \top$. We call such filters \emph{regular,} and we show that each regular filter on a pointless frame $E$ arises as $y_a$ for some frame $L$ having pointless part $E$ and maximal element $a$. Furthermore, the filters produced by distinct maximal elements are \emph{independent,} i.e., contain disjoint elements of $\pi L$. We refer to the family $W \equiv \setof{y_a}{a \in \max L}$ of such filters as the \emph{spatial support of $L$.} Together with the pointless part of $L$, $W$ determines the fat reflection of an atomless frame (Theorem \ref{Thm:3}). $W$ also determines whether or not $L$ is spatial, i.e., whether $\sigma_L$ is an isomorphism (Proposition \ref{Prop:22}), and whether or not $L$ is compact (Proposition \ref{Prop:21}).

In Section \ref{Sec:MaxRndFltr} we take up the situation that arises when all of the filters of the spatial support $W$ are maximal proper round filters. This assumption has the great advantage that it guarantees the complete regularity of the synthetic construction of a frame from its pointless part $E$ and its spatial support $W$. However, it has the disadvantage that the pointless part can grow bigger than $E$, at least if $W$ is taken to be the entire family of maximal proper round filters on $E$. In fact, in the latter case  we show in Proposition \ref{Prop:24} that the synthetically constructed frame is $\beta E$, the compact coreflection of $E$. That fact motivates a final digression, in which we establish the existence of what in spatial terms is a particularly exotic compactification in compact Tychonoff spaces without isolated points (Corollary \ref{Cor:4} and Proposition \ref{Prop:25}.)  

\section{Preliminaries}\label{Sec:Prelim}

For purposes of handy reference and to fix notation, we record here the few background results which underlie what follows. The material is almost entirely folklore and should be skipped upon a first reading, and then consulted only as necessary. 

\subsection{Naked frames, completely regular frames, and  witnessing families}

The context for these remarks is the category $\mbf{F}$ of completely regular frames with frame homomorphisms. Unless otherwise explicitly stipulated, all frames are assumed completely regular and all spaces are assumed Tychonoff.  A couple of constructions require an excursion into the category $\mbf{nF}$ of \emph{naked frames,} i.e., frames without the hypothesis of complete regularity, and then a return to $\mbf{F}$ by way of the completely regular coreflection. That is, we first construct a naked frame, labeled for instance $L'$, and then extract its largest completely regular subframe, labeled for instance $L$.

The issue of whether a particular naked frame is completely regular thus plays an important role in what follows, and our working definition of this important notion is as follows.  When speaking of two elements $a_i$ of a frame $L$, to say that \emph{$a_1$ lies completely below $a_2$} is to say that there is a family of witnesses $\ssetof{b_p}{\mbb{Q}} \subseteq L$ such that $a_1 \leq b_p \prec b_q \leq a_2$ for all $p < q$ in $\mbb{Q}$. We write $a_1 \combel a_2$, and we refer to $\ssetof{b_p}{\mbb{Q}}$ as a \emph{witnessing family for $a_1 \combel a_2$}.

\subsection{Nuclei and  prenuclei}

\begin{notation*}[nucleus notation]
	We make use of the following notational conventions for nuclei.
	\begin{itemize}
		\item 
		We use lower case Greek letters to denote nuclei on frames, often subscripted with the name of the frame, as in $\delta_L$.
		
		\item 
		We use primed lower case Greek letters to denote prenuclei on frames, often subscripted with the name of the frame, as in $\delta_L'$.  The corresponding nucleus is then designated by the same Greek letter without the prime.  See Lemma \ref{Lem:15}.
		
		\item 
		 We define and denote the \emph{kernel of a nucleus $\delta$ on a frame $L$} to be 
		 \[
			 \ker \delta
			 \equiv \setof{a \in L}{\delta(a) = \top}.
		 \]
		 See Lemma \ref{Lem:17}. 
		 
		 \item 
		 For a nucleus $\delta$ on a frame $L$, we denote its fixed point set, aka its  sublocale, by 
		 \[
			 \fix \delta
			 \equiv \setof{a \in L}{\delta(a) = a}.
		 \]
	\end{itemize}
\end{notation*}

\begin{lemma}\label{Lem:15}
	For a prenucleus $\delta'$ on a frame $L$, define for all $a \in L$ and for all ordinals $\beta$: 
	\begin{align*}
		\delta^\beta(a)
		&\equiv a,&&\text{if $\beta = 0$},\\
		\delta^\beta(a)
		&\equiv \delta' \circ \delta^\alpha(a)&&\text{if $\beta =  \alpha + 1$,}\\
		\delta^\beta(a)
		&\equiv \sbv{lr}{\alpha < \beta} \delta^\alpha(a)&&\text{if  $\beta$ is a limit ordinal,}\\
		\delta(a)
		&\equiv \delta^\beta(a)&&\text{for some (any) $\beta$ such  that $\delta^\beta(a) = \delta^{\beta + 1}(a)$.}
	\end{align*}
	Then $\delta$ is the unique nucleus on $L$ such that $\fix \delta' = \fix \delta$.
\end{lemma}

\begin{proof}
	See \cite[III 11.5]{PicadoPultr:2012}.
\end{proof}

\begin{lemma}\label{Lem:16}
	For a filter $F$ on a frame $L$, the map 
	\[
		\map{\delta_F'}{L}{L}
		\equiv \left(a \mapsto \bigvee_F(b \to a)\right), 
		\qquad a \in L,
	\]
	functions as a prenucleus on $L$.
\end{lemma}

\begin{proof}
	It is clear that $a \leq \delta_F'(a) \leq \delta_F'(b)$ for $a \leq b$ in $L$. And for elements $a_i \in L$ we have
	\begin{gather*}
		\delta'_F(a_1 \wedge a_2)
		= \bigvee_F (b \to (a_1 \wedge a_2))
		= \bigvee_F ((b \to a_1) \wedge (b \to a_2))\\
		 \geq \bigvee_F (a_1 \wedge (b \to a_2))
		= a_1 \wedge \bigvee_F (b \to a_2)
		= a_1 \wedge \delta_F'(a_2). \qedhere
	\end{gather*}
\end{proof}

\begin{definition*}[normal filter]
	A filter $F$ on a frame $L$ is said to be \emph{normal} if $\delta_F(a) \in F$ implies $a \in F$ for all $a \in L$. Here $\delta_F$ is the nucleus associated per Lemma \ref{Lem:15} with the prenucleus $\delta_F'$ of Lemma \ref{Lem:16}.
\end{definition*}

\begin{lemma}\label{Lem:17}
	A filter on a frame is the kernel of a nucleus if and only if it is normal. In detail, if $\delta$ is a nucleus on a frame $L$ then $\ker \delta$ is a normal filter, and if $F$ is a normal filter on $L$ then $\delta_F$ is the unique nucleus on $L$ for which $F = \ker \delta_F$.
\end{lemma}

\begin{proof}
	If $F = \fix \delta$ for a nucleus $\delta$ and if $\delta(a) \in F$ then $\delta \circ \delta(a) = \top$, and since $\delta$ is idempotent we get $\delta(a) = \top$ and $a \in F$. On the other hand, suppose that $F$ is a normal filter and $\delta_F'$ is the prenucleus defined from it as in Lemma \ref{Lem:16} and $\delta^\alpha$ is defined from $\delta_F'$ as in Lemma \ref{Lem:15}. Then since for any $a \in F$ we have 
	\[
		\delta_F (a) 
		\geq \delta_F^1(a)
		= \delta_F'(a)
		= \bigvee_F (b \to a)
		= \top,
	\]
	we see that $F \subseteq \ker \delta_F$. And for any $a \in \ker \delta_F$ we have that $a \in F$ since $\delta_F(a) = \top \in F$ and $F$ is normal.
\end{proof}

\begin{lemma}\label{Lem:5}
	Let $m$ and $n$ be frame homomorphisms with common domain $L$, and suppose that $m$ is surjective. Then $n$ factors through $m$ if and only if $m(a) = \top$ implies $n(a) = \top$ for all $a \in L$. 
	\[
	\begin{tikzcd}
		L \arrow{r}{m} \arrow{d}[swap]{n}
		&M \arrow[dotted]{dl}\\
		N&
	\end{tikzcd}
	\]
\end{lemma}

\begin{proof}
	Consider elements $a_i \in L$ such that $m(a_1) = m(a_2)$. Find a subset $B \subseteq L$ such that $\bigvee B = a_1$, and such that for each $b \in B$ we have $b \prec a_1$ as witnessed by another element $c_b$, i.e., $c_b \vee a_1 = \top$ and $c_b \wedge b = \bot$. Then the fact that $m(c_b \vee a_2) = m(c_b \vee a_1) = \top$ implies that $n(c_b) \vee n(a_2) = n(c_b \vee a_2) = \top$, combined with the fact that $n(c_b) \wedge n(b) = n(c_b \wedge b) = \bot$, yields $n(b) \prec n(a_2)$. In sum we have $n(a_1) = n\left(\bigvee B\right) = \bigvee_B n(b) \leq n(a_2)$, and a symmetrical argument gives $n(a_2) \leq n(a_1)$.
\end{proof}

\begin{lemma}\label{Lem:14}
	Let $\map{m}{L}{M}$ be a frame homomorphism, and let $\delta_L$ and $\delta_M$ be nuclei on $L$ and $M$, respectively. Then $m$ drops through $\delta_L$ and $\delta_M$, i.e., there exists a unique map $\bar{m}$ such that $\delta_M \circ m = \bar{m} \circ \delta_L$, if and only if $m(\ker L) \subseteq \ker M$.
	\[
	\begin{tikzcd}
		L \arrow{r}{m} \arrow{d}[swap]{\delta_L}
		& M \arrow{d}{\delta_M}\\
		\fix \delta_L \arrow{r}[swap]{\bar{m}}
		&\fix \delta_M
	\end{tikzcd}
	\]
\end{lemma}

\begin{proof}
	This follows immediately from Lemma \ref{Lem:5}
\end{proof}
\subsection{Congruences}

\begin{notation*}[congruence notation]
	We make use of the following notational conventions for frame congruences.
	\begin{itemize}
		\item 
		We use capital Greek letters to denote congruences on a frame $L$, and we denote the congruence frame itself, aka the assembly of $L$, by $\con L$.
		
		\item 
		($\Phi_a$, $o_a$, $O_a$) We denote the congruence of the open quotient associated with an element $a$ of a frame $L$ by
		\[
			\Phi_a
			\equiv \setof{(a_1, a_2)}{a \wedge a_1 = a \wedge a_2}, 
		\]
		and the quotient map by $\map{o_a}{L}{L/\Phi_a}$. We denote the quotient frame $L/\Phi_a$ by $O_a$, and often identify it with its sublocale $\{a \to b : b \in L\}$, in which case we also identify the map $o_a$ with (the range restriction of) its nucleus $(b \mapsto a \to b)$, $b \in L$.
		
		\item 
		We denote the congruence of the closed quotient associated with an element $a$ of a frame $L$ by
		\[
			\Psi_a
			\equiv \setof{(a_1, a_2)}{a \vee a_1 = a \vee a_2}, 
		\]
		and the quotient map by $\map{c_a}{L}{L/\Psi_a}$. We denote the quotient frame $L/\Psi_a$ by $C_a$, and often identify it with its sublocale $\upset{a} \equiv \setof{b}{b \geq a}$, in which case we also identify the map $c_a$ with (the range restriction of) its nucleus $(b \mapsto a \vee b)$, $b \in L$. 
		
		\item 
		We denote the coarsest dense congruence on a frame $L$ by 
		\[
			\Delta_L
			\equiv \setof{(a_1, a_2)}{a_1\st = a_2\st},
		\]
		and the quotient map by $\map{\delta_L}{L}{L/\Delta_L} \equiv \Delta L$. We often identify the quotient with its sublocale $L\stst \equiv \ssetof{a\stst}{L}$, in which case we also identify the map $\delta_L$ with (the range restriction of) its nucleus $(a \mapsto a\stst)$, $a \in L$. We refer to the map or it codomain as the \emph{skeleton of $L$}.
		
		\item 
		We denote the congruence of a frame surjection $\map{m}{L}{M}$ by 
		\[
			\Theta_m
			\equiv \setof{(a_1, a_2)}{m(a_1) = m(a_2)}.
		\]
	\end{itemize} 
\end{notation*}

\begin{lemma}\label{Lem:11}
	For elements $a \leq b$ and congruence $\Xi$ on a frame $L$,
	\[
		(a,b) \in \Xi 
		\iff \Phi_a \wedge \Psi_b \leq \Xi.
	\]
	In particular,
	\[
		(a, \top) \in \Xi \iff \Phi_a \leq \Xi 
		\quad\text{and}\quad
		(\bot, b) \in \Xi \iff \Psi_a \leq \Xi.
	\]
\end{lemma}

\begin{lemma}\label{Lem:8}
	For any frame surjection $\map{m}{L}{M}$,
	\[
		\Theta_m 
		\geq \sbv{lr}{m^{-1}(\top)}\Phi_a \vee \sbv{lr}{m^{-1}(\bot)}\Psi_a.
	\]
\end{lemma}

\begin{proof}
	If $a \in m^{-1}(\top)$ and $(a_1, a_2) \in \Phi_a$ then 
	\begin{align*}
	m(a_1)
	&= m(a_1) \wedge \top
	= m(a_1) \wedge m(a)
	= m(a_1 \wedge a)
	= m(a_2 \wedge a)\\
	&= m(a_2) \wedge m(a)
	= m(a_2) \wedge \top
	= m(a_2). \qedhere
	\end{align*}
\end{proof}

\begin{lemma}\label{Lem:9}
	Let $\map{m}{L}{M}$ be a frame surjection. 
	\begin{enumerate}
		\item 
		The map 
		\begin{align*}
			&\map{m^{-1}}{\con M}{\upset{\Theta_m}_{\con L}}\\
			&= \big(\Delta \mapsto m^{-1}(\Delta) 
			\equiv \setof{(a_1, a_2)}{(m(a_1), m(a_2)) \in \Delta}\big)
		\end{align*}
		is a frame isomorphism.
		
		\item 
		For any element $b \in M$, $m^{-1}(\Phi_b) = \Theta_m \vee \Phi_{m_*(b)}$.
	\end{enumerate}
\end{lemma}

\begin{proof}
	(2) To show that $\Theta_m \vee \Phi_{m_*(b)} \leq m^{-1}(\Phi_b)$, first note that for elements $a_i \in L$ such that $(a_1, a_2) \in \Theta_m$ we have
	\begin{align*}
		m(a_1) = m(a_2)
		&\implies m(a_1) \wedge b = m(a_2) \wedge b
		\iff (m(a_1), m(a_2)) \in \Phi_b \\
		&\iff (a_1, a_2) \in m^{-1}(\Phi_b).
	\end{align*}
	Then observe that for any elements $a_i \in L$ such that $(a_1, a_2) \in \Phi_{m_*(b)}$ we also have
	\begin{gather*}
		a_1 \wedge m_*(b) = a_2 \wedge m_*(b) \implies \\
		m(a_1) \wedge b = m(a_1) \wedge m \circ m_*(b) = m(a_2) \wedge m \circ m_*(b) =
		 m(a_2) \wedge b\\
		\iff (m(a_1), m(a_2)) \in \Phi_b
		\iff (a_1, a_2) \in m^{-1}(\Phi_b).
	\end{gather*}
	To show that $\Theta_m \vee \Phi_{m_*(b)} \geq m^{-1}(\Phi_b)$, observe that for any pair of elements $a_i \in L$ such that $(a_1, a_2) \in m^{-1}(\Phi_b)$ we have
	\begin{gather*}
		(m(a_1), m(a_2)) \in \Phi_b 
		\iff m(a_1) \wedge b = m(a_2) \wedge b \implies\\
		m_*\circ m(a_1) \wedge m_*(b) = m_* \circ m(a_2) \wedge m_*(b) \iff \\
		(m_* \circ m(a_1), m_* \circ m(a_2)) \in \Phi_{m_*(b)}.
	\end{gather*}
	Since $(a_i, m_* \circ m(a_i)) \in \Theta_m$, we can conclude from transitivity that $(a_1, a_2) \in \Theta_m \vee \Phi_{m_*(b)}$.
\end{proof}

\subsection{Successors and predecessors, atoms and maximal elements}

\begin{definition*}[successor, predecessor, $a^+$, atom, maximal element, $\max L$, pointless frame, interpolative lattice]
	When speaking of two elements $a$ and $c$ of a distributive lattice $L$, we say that \emph{$c$ is a successor of $a$}, or that \emph{$a$ is a predecessor of $c$}, if $c > a$ and for any element $b$ such that $a \leq b \leq c$, either $b = a$ or $b = c$. We denote the set comprised of $a$ together with its successors by  
	\[
		a^+
		\equiv
		\{a\} \cup \setof{c}{\text{$c$ is a successor of $a$}}.
	\]
	A successor of $\bot$ is called an \emph{atom of $L$,} and an element having $\top$ as a successor is called \emph{maximal}. We denote the set of maximal elements of a frame $L$ by 
	\[
		\max L
		\equiv \setof{a}{\text{$a$ is maximal in $L$}}.
	\] 
	A frame $L$ is called \emph{pointless} if $\max L = \emptyset$. Thus the \enquote{empty frame} $\{\bot = \top\}$, i.e., the topology of the empty space, is pointless, whereas the two element frame $\mbb{2} = \{\bot \neq \top\}$, i.e., the topology of the singleton space, is not. A lattice $L$ is called \emph{interpolative} if it has no successors or predecessors, i.e., if for all elements $a < c$ there exists an element $b$ such that $a < b < c$.
\end{definition*}

\begin{lemma}\label{Lem:1}
	The following hold for elements $a$ and $b$ in a distributive lattice $L$.
	\begin{enumerate}
		\item
		The maps
		\begin{align*}
			[a \wedge b, a] \ni c
			&\longrightarrow c \vee b \ \ \ \ \text{and}\\
			c \wedge a
			&\longleftarrow c \in [b, a\vee b]
		\end{align*} 
		are inverse lattice isomorphisms.
		
		\item 
		$a \wedge b^+ \subseteq (a \wedge b)^+$.		
	\end{enumerate}
\end{lemma}

\begin{proof}
	(2) Consider a successor $c$ of $b$. Since 
	\[
		b \leq (a \wedge c) \vee b
		= (a \vee b) \wedge c
		\leq c,
	\]
	either $(a \wedge c) \vee b = b$ or $(a \vee b) \wedge c = c$. The first possibility implies that $a \wedge c \leq b$ and hence $a \wedge c = a \wedge b$, while the second implies that $a \vee b \geq c$ and hence that $a \wedge c$ is a successor of $a \wedge b$ by part (1). In either case we get that $a \wedge c \in (a \wedge b)^+$.
\end{proof}

\begin{lemma}\label{Lem:3}
	Let $\map{m}{L}{M}$ be a frame homomorphism. 
	\begin{enumerate}
		\item 
		If $b$ is prime (aka meet irreducible) in $M$ then $m_*(b)$ is prime in $L$.
		
		\item 
		If $m$ is surjective then $m_*(\max M) \subseteq \max L$. 
		
		\item
		If $m$ is surjective and $a$ is prime in $L$ then $m(a)$ is prime in $M$.
	\end{enumerate}
\end{lemma}

\begin{proof}
	(3) Let $a$ be a prime element of $L$ such that $m(a) < \top$, and consider elements $c_i \in M$ such that $c_1 \wedge c_2 = m(a)$. Let $a_i \equiv m_*(c_i)$, $i = 1,2$, and note that 
	\[
	a_1 \wedge a_2
	= m_*(c_1) \wedge m_*(c_2)
	= m_*(c_1 \wedge c_2)
	= m_* \circ m(a)
	\geq a
	\]
	Now $a$ is maximal in $L$ because $L$ is regular, so $m_* \circ m(a)$ is either $a$ or $\top$. The latter possibility is ruled out by the fact that $m \circ m_* \circ m(a) = m(a)$, so we have $a_1 \wedge a_2 = a$. By the primeness of $a$, then, either $a_1 \leq a$ or $a_2 \leq a$, hence $m(a_1) \leq m(a)$ or $m(a_2) \leq m(a)$. But since $m$ is surjective, $m(a_i) = m \circ m_*(c_i) = c_i$, so either $c_1 \leq m(a)$ or $c_2 \leq m(a)$, i.e., $m(a)$ is prime in $M$. 
\end{proof}

Lemma \ref{Lem:2} explains, among other things, how complemented successor elements arise from  maximal elements in frame factors.  

\begin{lemma}\label{Lem:2}
	The following hold in a frame $L$.
	\begin{enumerate}
		\item
		An element is a predecessor of $\top$ if and only if it is maximal if and only if it is prime and unequal to $\top$.
		
		\item 
		An element of a frame is an atom if and only if it is the complement of a maximal element. 
		
		\item 
		An element $a$ is a successor of an element $b$ if and only if $a > b$ and $a \to b$ is a maximal element of $L$.
		
		\item 
		An element $b$ is a predecessor of an element $a$ if and only if $b = a \wedge c$ for a maximal element $c \ngeq a$.  
		
		\item 
		If $L = M \times N$ and $a$ is a maximal element of $M$ then $(\top, \bot)$ is a successor of $(a, \bot)$ and $(\top, \bot)$ has complement $(\bot, \top)$. In this case $(\top, b)$ is a successor of $(a, b)$ for all $b \in N$. 		
		
		\item 
		Every complemented successor in $L$ arises as in (3). That is, if $c$ is a successor of $a$ in $L$ and $c$ is complemented then the map 
		\[
			\map{m}{L}{\downset{c} \times \downset{c^*}}
			= (b \mapsto (b \wedge c, b \wedge c^*))
		\]
		is an isomorphism, and $m(c) = (c, \bot)$ is a successor of $m(a) = (a, \bot)$. 
	\end{enumerate}
\end{lemma}

\begin{proof}
	The assumption that $L$ is regular implies the equivalence of maximality and primeness in (1). Likewise in (2), since an atom $a \in L$ is the join of a set $A$ of elements rather below it, there must be at least one element $b \in A$ such that $b > \bot$. This implies that $b = a$, but more to the point, that $a$ is rather below itself, hence $a$ is complemented. 
	
	(3)  If $a$ is a successor of $b$ then $b$ is maximal in $\downset{a}$, and since the map $\downset{a} \to O_a = (c \mapsto a \to c)$ is a frame isomorphism, $a \to b$ is maximal in $O_a$. Thus $a \to b$ is prime in $O_a$ by (1), and $o_{a*}(a \to b) = a\to b$ is prime in $L$ by Lemma \ref{Lem:3}(1). 
	
	On the other hand, if $a > b$ and $a \to b \in \max L$ then $a \vee (a \to b) = \top$ since $a \nleq a \to b$. Therefore Lemma \ref{Lem:1}(1) provides an isomorphism between the intervals $[b, a]$ and $[a \to b, \top]$, from which we see that $a$ is a successor of $b$.  
\end{proof}

\begin{lemma}\label{Lem:10}
	For any frame $L$, $\max (\con L) = \setof{\Psi_b}{b \in \max L}$. Consequently the set of atoms of $\con L$ is $\ssetof{\Phi_b}{\max L}$.
\end{lemma}

\begin{proof}
	For an element $b \in \max L$, consider a pair $(a_1, a_2)\notin \Psi_b$, so that $a_i \nleq b \geq a_j$ for $i \neq j$, say $a_1 \nleq b \geq a_2$. Then any congruence which contains both $\Psi_b$ and $(a_1, a_2)$ must contain every pair $(a_3, a_4)$.  For if $(a_3, a_4) \notin \Psi_b$, say $a_3 \nleq b \geq a_4$, then $\Psi_b$ contains both $(a_1, a_3)$ and $(a_2, a_4)$, hence also $(a_3, a_1)$ by symmetry, and then finally $(a_3, a_4)$ by transitivity. We have proven that $\Psi_b \in \max (\con L)$.
	
	On the other hand, suppose $\Xi \in \max (\con L)$. Then $L/\Xi$ is isomorphic to $\mbb{2} \equiv \{\bot, \top\}$, and if $\map{q}{L}{\mbb{2}}$ is the quotient map then $b \equiv \bigvee q^{-1}(\bot) \in \max L$ is such that $\Xi = \Psi_b$. The second sentence of the lemma follows from the first via Lemma \ref{Lem:2}(2). 	
\end{proof}

\begin{lemma}\label{Lem:4}
	A frame $L$ is pointless if and only if it is interpolative.
\end{lemma}

\begin{proof}
	An interpolative frame is clearly pointless. And in any frame $L$, if elements $a < c$ admit no element $b$ such that $a < b < c$ then $a$ is  maximal in $\downset{c}$, and since $\downset{c} \approx O_c = \ssetof{c \to b}{L}$, it follows that $c \to a$ is maximal in $O_c$. Consequently  $o_{c*}(c \to a) = c \to a$ is maximal in $L$ by Lemma \ref{Lem:3}(1).
\end{proof}

\subsection{Round and regular filters}

\begin{definition*}[filter, independent family of filters, round filter, regular filter, $\regfltr L$, round ideal]
	A filter on a frame $L$ is a nonempty upset $F \subseteq L$ which is closed under binary meets. The filter is said to be \emph{proper} if it does not contain $\bot$. A family of filters is said to be \emph{independent} if any two distinct filters of the family contain disjoint elements. A filter $F$ is said to be \emph{round} if for every element $a \in F$ there exists an element $b \in F$ such that $b \combel a$. Round ideals are defined dually. A filter $F$ is said to be \emph{regular} if it is round and $\bigvee_F b^* = \top$. We denote the family of proper regular filters on a frame $L$ by $\regfltr L$.
\end{definition*}

\begin{notation*}[$x_a$]
	We use lowercase letters near the end of the Latin alphabet, e.g., x, y, z, to designate filters on a frame. For a maximal element $a$, we denote the filter $\setof{b}{b \nleq a}$ by subscript, e.g., $x_a$, $y_a$, $z_a$. We use uppercase letters near the end of the alphabet, e.g., W, X, Y, Z, to designate families of filters on a frame. 
\end{notation*}

Lemma \ref{Lem:22} records the basic information about round and regular filters.  

\begin{lemma}\label{Lem:22}
	The following hold in any frame $L$.
	\begin{enumerate}
		\item 
		Every proper round filter is contained in a maximal proper round filter.
		
		\item 
		For every proper filter $x \subseteq L$,
		\[
			\mathring{x}
			\equiv \setof{a \in x}{\exists b \in x\ (b \combel a)}
		\]
		is the largest round filter contained in $x$.
		
		\item 
		A proper round filter $x$ is maximal if and only if for every element $a \notin x$ and every element $b \combel a$ there exists an element $c \in x$ such that $b \wedge c = \bot$. 
		
		\item 
		Distinct maximal proper round filters contain disjoint elements. Therefore any family of maximal proper round filters is independent.		

		\item 
		If $x$ is a maximal proper round filter then $\bigvee_x b^*$ is a prime element of $L$. That is, either $x$ is regular, i.e., $\bigvee_x b^* = \top$, or $\bigvee_x b^* = a \in \max L$. In the latter case $x = x_a$.
				
		\item 
		If $a \in \max L$ then $x_a$ is a maximal proper round filter such that $\bigvee_{x_a} b\st = a$. 
		
		\item 
		For any $a \in \max L$, $x_a$ is completely prime, i.e., $\bigvee A \in x_a$ implies $A \cap x_a \ne \emptyset$ for any subset $A \subseteq L$.
		
		\item 
		A proper round filter is maximal if and only if it is of the form $\mathring{y}$ for some ultrafilter $y$ on $L$.
		
		\item 
		A maximal proper round filter $x$ on $L$ has the following sort of primeness:
		\[
			(a_i \combel b_i \text{ and } a_1 \vee a_2 \in x)  \implies (b_1 \in x \text{ or }b_2 \in x).
		\]
		Conversely, any proper round filter having this sort of primeness is maximal (among proper round filters). 
		
		\item 
		The maps 
		\begin{align*}
			x 
			&\longrightarrow \langle a\st : a \in x \rangle_\text{idl} \quad \text{and}\\
			\langle b\st : b \in I \rangle_\text{fltr}
			&\longleftarrow I			  
		\end{align*}
		constitute inverse order preserving bijections between the sets of round filters and round ideals on $L$.  
	\end{enumerate}
\end{lemma}

\begin{proof}
	(3)  Any proper round filter $x$ satisfying this condition is clearly maximal, and if a proper round filter $x$ and element $a \notin x$ violate this condition, i.e., if there exists an element $b \combel a$ such that $b \wedge c > \bot$ for all $c \in x$, then the filter $y$ generated by $x \cup \{b\}$ has the feature that $\mathring{y}$ is a proper round filter properly containing $x$. 
	
	(4) If maximal round filters $x_i$ are distinct then there exist elements $a_i \in x_i \smallsetminus x_j$ and elements $b_i \in x_i$ such that $b_i \combel a_i$. By (3) there exist elements $c_i \in x_i$ such that $c_i \wedge b_j = \bot$. But then $c_i \wedge b_i \in x_i$ and $(c_1 \wedge b_1) \wedge (c_2 \wedge b_2) = \bot$.   
	
	(5) Suppose that $x$ is a maximal proper round filter, and for the sake of argument suppose that $\bigvee_x b^* = a < c < \top$. We claim that $c \notin x$, for otherwise there exists an element $b \in x$ such that $b \combel c$. But if $d$ witnesses $b \prec c$, i.e., if $b \wedge d = \bot$ and $c \vee d = \top$, then we get $d \leq b^* \leq a \leq c$, resulting in the contradiction $c = c \vee d = \top$. 
	
	Since $a < c$ there exists an element $d \combel c$ such that $d \nleq a$, whereupon part (3) above produces an element $b \in x$ such that $b \wedge d = \bot$. We are led to the contradiction $d \leq b^* \leq a$.   We conclude that $\bigvee_x b\st$ is prime. 
	
	If $b \in x$ then since $x$ is round there exists an element $c \in x$ such that $c \prec b$. Let $d$ witness $c \prec b$, i.e., $c \wedge d = \bot$ and $d \vee b = \top$, so that $d \leq c^* \leq a$. Then $b \nleq a$, for otherwise $\top = d \vee b \leq a$, contrary to assumption. On the other hand, if $b \notin x$ then every element $c \combel b$ is disjoint from some element $d \in x$, hence $c \leq d^* \leq a$. It follows that $b = \bigvee \setof{c}{c \combel b} \leq a$.
	
	(6) If $a \in \max L$ and $b \in x_a$ then, since $b = \bigvee C$ for $C \equiv \setof{c}{c \combel b}$, it must be the case that $C \cap x_a \neq \emptyset$, for otherwise the fact that $C \subseteq \downset{a}$ would lead to the contradiction $b = \bigvee C \in \downset{a}$. Thus $x_a$ is round. To see that $x_a$ is maximal among proper round filters, consider elements $d \combel c \notin x_a$. Then $c \leq a$, so there exists an element $b$ witnessing $d \prec a$, i.e., $b \wedge d = \bot$ and $b \vee a = \top$. Then $b \nleq a$, for otherwise we would have $a = b \vee a = \top$, contrary to hypothesis. That is, $b \in x_a$. The fact that $\bigvee_{x_a}b\st = a$ now follows from (5).
	
	(7)  If for some subset $A \subseteq L$ and element $a \in \max L$ we have $\bigvee A = b \in x_a$ then $A \nsubseteq \downset{a}$, for otherwise $b = \bigvee A \leq a$, contrary to assumption.
	
	(8) If a proper round filter $x$ is maximal then it is the largest round filter contained in any filter containing it. And if $x = \mathring{y}$ for some ultrafilter $y$ on $L$ and if $b \combel a \notin x$ then there must exist an element $b_1$ such that $b \combel b_1 \combel a$, for which we know that $b, b_1 \notin y$ by (3). By virtue of the maximality of $y$, therefore, $b^*, b_1^* \in y$, and since $b_1^* \combel b^*$ we get $b^* \in x$ by (2). We can conclude that $x$ is maximal among proper round filters by (3).  
	
	(9) Let $x$ be a maximal proper round filter and let $y$ be an ultrafilter for which $x = \mathring{y}$, and suppose that $a_i \combel b_i$. If $a_1 \vee a_2 \in x$ then $a_1 \vee a_2 \in y$, hence either $a_1 \in y$ or $a_2 \in y$, with the result that either $b_1 \in \mathring{y} = x$ or $b_2 \in \mathring{y} = x$.
	
	Conversely, suppose a proper round filter $x$ enjoys this sort of primeness, and consider elements $b \combel a \notin x$ with witnessing family $\ssetof{c_p}{\mbb{Q}}$ for $b \combel a$. Fix $p < q < r$ in $\mbb{Q}$.  Then $c_q^* \vee c_r = \top \in x$, and since $c_q^* \combel c_p^*$ with witnessing family $\ssetof{c_s^*}{p < s < q}$ and $c_r \combel a$ with witnessing family $\ssetof{c_s}{r < s}$, the primeness of $x$ yields $c_p^* \in x$ or $a \in x$, and the latter condition is ruled out by assumption. We conclude that $x$ is maximal among proper round filters by (3).  
	
	(10) This follows from the fact that for elements $a$ and $b$ in any frame, $a \combel b$ if and only if $b\st \combel a\st$.
\end{proof}

An immediate consequence of Lemma \ref{Lem:22}(5) will be important in what follows. 

\begin{corollary}\label{Cor:3}
	A maximal proper round filter on a pointless frame is regular.
\end{corollary}

\begin{lemma}\label{Lem:23}
	Let $\map{m}{L}{M}$ be a frame surjection.
	\begin{enumerate}
		\item 
		If $x$ is a round filter on $L$ then $m(x)$ generates a round filter on $M$. 
		
		\item 
		The filter generated by $m(x)$ is proper if and only if $m_*(\bot) \notin x$. 
		
		\item 
		If $x$ is a maximal proper round filter on $L$ and $y$ is a maximal proper round filter on $M$ containing $m(x)$ then $m^{-1}(y)\mathring{} = x$.
		
		\item 
		If $x$ is a regular filter on $L$ and $m$ is dense then $m(x)$ is a regular filter on $M$.
		
		\item 
		If $x = x_a$ for some $a \in \max L$ such that $m(a) = \top$, and if $m$ is dense, then $m(x)$ is a regular filter on $M$. 
		
		\item 
		(Converse of (5)) If $y$ is a regular filter on a frame $M$ then there is a frame $L$ admitting a dense surjection $\map{m}{L}{M}$ and having a maximal element $a \in \max L$ such that $m(x_a) = y$ and $m(a) = \top$.
	\end{enumerate}
\end{lemma}

\begin{proof}
	(4) If $\top > a \in M$ then $\top > m_*(a) \in L$, and since $x$ is regular there exists an element $b \in x$ such that $b^* \nleq m_*(a)$. It follows that $m(b^*) \nleq a$, and since $m$ is dense, $m(b^*) = m(b)^*$. 
	
	(5) Consider an element $\top > c \in M$. Then $\top > m_*(c) \in L$, and we claim that $m_*(c) \ngeq b^*$ for some $b \in x_a$.  For otherwise $m_*(c) \geq \bigvee_{x_a} b^* = a$ by parts (5) and (6) of Lemma \ref{Lem:22}, in which case we would have $c = m \circ m_*(c) \geq m(a) = \top$, contrary to assumption. As before, it follows from the claim that $m(b^*) \nleq c$, and also as before $m(b^*) = m(b)^*$ because $m$ is dense.   
	
	(6) This is essentially the content of \cite[4.2.1]{Ball:2014}. Unfortunately, the result is incorrect as its stands, inasmuch as the crucial hypothesis of roundness is omitted from the definition of regular filter.\footnote{The error in the proof occurs in the second paragraph, starting with the sentence which begins \enquote{On the other hand it is obvious that \dots}} 
	Fortunately, the error does not invalidate the subsequent results in that article, and with the aid of the missing hypothesis of roundness, the reader will have no difficulty supplying a correct proof.  
\end{proof}

\begin{corollary}
	For any frame surjection $\map{m}{L}{M}$, the set $\ssetof{m(x_a)}{\max L}$ is an independent family of regular filters on $M$.
\end{corollary}

\begin{proof}
	The family $\ssetof{x_a}{\max L}$ is independent by Lemma \ref{Lem:22}(4). Since $m$ is dense, $\ssetof{m(x_a)}{\max L}$ is independent as well.
\end{proof}

%\section{A presentation of the assembly}
%
%The symbol $L$ will stand for a frame throughout this section. 
%\subsection*{The meet semilattice $D$ of differences}
%
%\begin{definition*}[$\pmap{a}{b}{L}$]
%	For elements $a \leq b$ in $L$, we denote the frame of elements between $a$ and $b$ by $[a, b] \equiv\setof{c}{a \leq c \leq b}$, and we denote the corresponding projection frame homomorphism by
%	\[
%		\pmap{a}{b}{L}
%		\equiv (c \mapsto \big(c \vee a) \wedge b = (c \wedge b) \vee a\big).
%	\]
%\end{definition*}
%
%\begin{definition*}
%	We denote the set of ordered pairs from $L$ by 
%	\[
%		D'
%		\equiv \setof{(a, b) \in L^2}{a \leq b},
%	\]
%	and we preorder this family by declaring 
%	\[
%		(a, b) \ll (c, d) 
%		\qtq{iff} d \vee a \geq b 
%		\text{ and } c \wedge b \geq a. 
%	\]
%	We denote the induced equivalence by 
%	\[
%		(a, b) \sim (c, d)
%		\iff (a, b) \ll (c, d)
%		\qtq{and} (c, d) \ll (a, b),
%	\]
%	the equivalence class of $(a, b) \in D'$ by 
%	\[
%		\langle a, b \rangle
%		\equiv \setof{(c,d) \in D}{(a,b) \sim (c,d)}, 
%	\]
%	and the set of equivalence classes by 
%	\[
%		D
%		\equiv \setof{\langle a, b\rangle}{(a,b) \in D'}.
%	\]
%	We abuse the notation to the extent of using $\ll$ to denote the induced partial order on $D$. We occasionalloy often refer to a member of $D$ by using a single letter, say $d$, and when we do so we shall use $\plower{d}$ and $\pupper{d}$ to designate  two elements of $L$ for which $d = \langle \plower{d}, \pupper{d}\rangle$. 
%\end{definition*}
%
%\begin{lemma}
%	$D$ is a meet semilattice with 
%	\begin{align*}
%		[a,b] \wedge [c, d] = []
%	\end{align*}
%\end{lemma}

\section{Two nuclei}\label{Sec:2Nuc}
	
	In the section we introduce and investigate the nuclei of interest. 
	
\subsection{The nucleus $\sigma$ and the spatial part of a frame}

\begin{lemma}
	On a frame $L$, the map 
	\[
		\map{\sigma_L}{L}{L}
		\equiv \left(a \mapsto \bigwedge \upset{a}_{\max L} \right)
	\]
	is a nucleus with fixed point set  
	\[
		\fix \sigma_L
		= \setof{a}{a = \bigwedge \upset{a}_{\max L}}
		\equiv \sigma L
	\]
	and kernel
	\[
		\ker \sigma_L
		= \setof{a}{\upset{a}_{\max L} = \emptyset}.
	\]
\end{lemma}
	
\begin{proof}
	The map $\sigma_L$ is clearly increasing and order preserving; to verify that it preserves binary meets, simply compute
	\begin{gather*}
		\sigma_L(a_1) \wedge \sigma_L(a_2)
		= \bigwedge \upset{a_1}_{\max L} \wedge \bigwedge \upset{a_2}_{\max L} =\\
		\bigwedge \big(\upset{a_1}_{\max L} \wedge \upset{a_2}_{\max L}\big)
		= \bigwedge \upset{a_1 \wedge a_2}_{\max L}
		= \sigma_L(a_1 \wedge a_2).
	\end{gather*}
	The penultimate equality is attributable to the primeness of maximal elements.
\end{proof}

Recall that a frame is called spatial if every element is the meet of the maximal elements above it. See \cite[II 5.3]{PicadoPultr:2012}.

\begin{proposition}
	The following hold for any frame $L$.
	\begin{enumerate}
		\item 
		$\sigma L$ is spatial.
		
		\item 
		$L$ is spatial if and only if $\sigma_L$ is an isomorphism if and only if $\ker \sigma_L = \{\top\}$ if and only if 
		\[
			\forall a < \top\, \exists b \in \max L\ (a \leq b).
		\]
		
		\item 
		$L$ is pointless if and only if $\sigma L$ is empty, i.e., $\sigma L = \{\top\}$.
	\end{enumerate}  
\end{proposition}

\begin{proof}
	The nucleus $\sigma_L$ fixes each maximal element, and each maximal element of $L$ remains maximal in $\sigma L$.
\end{proof}

\begin{definition*}[spatial part of a frame, $\mbf{sF}$]
	For a frame $L$, we refer to the map $\map{\sigma_L}{L}{\sigma L}$ or to its codomain $\sigma L$ as the \emph{spatial part of $L$.} We designate the full subcategory of $\mbf{F}$ comprised of the spatial frames by $\mbf{sF}$.
\end{definition*}

\begin{lemma}\label{Lem:19}
	Any frame homomorphism $\map{m}{L}{M}$ drops through $\sigma_L$ and $\sigma_M$, i.e., there is a unique frame homomorphism $\bar{m}$ such that $\sigma_m \circ m = \bar{m} \circ \sigma_L$. 
	\[
	\begin{tikzcd}
		L \arrow{r}{m} \arrow{d}[swap]{\sigma_L}
		&M \arrow{d}{\sigma_M}\\
		\sigma L \arrow{r}[swap]{\bar{m}}
		&\sigma M
	\end{tikzcd}.
	\]
\end{lemma}

\begin{proof}
	According to Lemma \ref{Lem:14}, it is sufficient to show that $m(a) \in \ker \sigma_M$ for any $a \in \ker \sigma_L$. But this is clearly so, for if 
	\[
		m(a) \notin \ker \sigma_M 
		= \setof{b \in M}{\upset{b}_{\max M} = \emptyset} 
	\]
	it is only because $m(a) \geq b$ for some $b \in \max M$, in which case $m_*(b) \in \upset{a}_{\max L}$, contrary to assumption.  
\end{proof}

\begin{proposition}\label{Prop:4}
	$\mbf{sF}$ is epireflective in $\mbf{F}$, and a reflector for the frame $L$ is its spatial part $\map{\sigma_L}{L}{\sigma L}$.
\end{proposition}

\begin{proof}
	This follows directly from Lemma \ref{Lem:19}.
\end{proof}

\subsection{The nucleus $\pi$ and the pointless part of a frame}

\begin{proposition}\label{Prop:3}
	In a frame $L$, the map 
	\[
		\pi_L'(a)
		\equiv \bigvee a^+
	\]
	is a prenucleus, with fixed point set 
	\[
		\pi L
		\equiv \setof{a}{\text{$a$ has no successor}}
		= \setof{a}{\forall c > a \: \exists b\ (a < b < c)},
	\]
	and with kernel being the normal filter generated by $\max L$, namely
	\[
		\ker \pi_L
		= \setof{a}{\forall b \geq a\ (b < \top \implies \text{$b$ has a successor})}.
	\]
\end{proposition}

\begin{proof}
	The map $\pi_L'$ is increasing by construction. To show that $\pi_L'$ is order-pre\-ser\-ving, consider $a \leq b$ in $L$. Then for each $c \in a^+$, either $c \leq b$ or $b \vee c \in b^+$ by Lemma \ref{Lem:1}, with the result that 
	\[
		\pi_L'(a) = \bigvee a^+
		\leq \bigvee b^+
		= \pi_L'(b).
	\]
	To check that $\pi_L' (a \wedge b) \geq a \wedge \pi_L'(b)$, simply observe that by Lemma \ref{Lem:1}(2) we have
	\[
		\pi_L'(a \wedge b) 
		= \bigvee (a \wedge b)^+
		\geq \bigvee (a \wedge b^+)
		= a \wedge \bigvee b^+
		= a \wedge \pi_L'(b). 
	\]	
	
	Certainly $\max L \subseteq \ker \pi_L$, i.e., $\pi_L(a) = \top$ for each $a \in \max L$; after all, $\top \in a^+$. This establishes that $\delta_{\max L} \leq \pi_L$, where $\delta_{\max L}$ represents the nucleus corresponding as in Lemmas \ref{Lem:15} and \ref{Lem:16} to the filter generated by $\max L$. We claim that $\delta_{\max L}(a) \geq c$ whenever $c$ is a successor of $a$. For in that case $b \equiv c \to a \in \max L$ by Lemma  \ref{Lem:2}(3), hence 
	\begin{align*}
		\top
		&= \delta_{\max L}(c \to a)
		\leq \delta_{\max L}(c) \to \delta_{\max L}(a)\\
		&\implies c \leq \delta_{\max L}(c) \leq \delta_{\max L}(a).
	\end{align*}
	The claim implies that $\pi_L \leq \delta_{\max L}$.
\end{proof}

The fact that $\ker \pi_L$ is generated by the maximal elements of $L$ (Proposition \ref{Prop:3}) implies that the congruence of $\pi_L$ is the join of the open congruences of the elements of $\max L$.
 
\begin{corollary}\label{Cor:2}
	For any frame $L$, 
	\[
		\Theta_{\pi_L}
		= \sbv{lr}{\max L} \Phi_a.
	\]
	In terms of sublocales,
	\[
		\pi L
		= \fix \pi_L
		= \bigcap_{\max L} O_a.
	\]
\end{corollary}

\begin{definition*}[pointless part of a frame]
	For a frame $L$, we refer to the map $\map{\pi_L}{L}{\pi L}$ or its codomain as the \emph{pointless part of $L$.}
\end{definition*}

The pointless part of a frame can be characterized in terms of its maximal elements.

\begin{proposition}\label{Prop:13}
	For any frame $L$,
	\[
	\pi L
	= \setof{a}{\forall c\ (a \leq c \in \max L \implies c \to  a  = a)}.
	\]
\end{proposition}

\begin{proof}
	Suppose that $a$ does not lie in the set displayed on the right, say $a \leq c$ for some $c \in \max L$ such that $c \to a > a$. Then $c \to a$ is a successor of $a$ by Lemma \ref{Lem:2}(3) because $(c \to a) \to a = c$ by virtue of the maximality of $c$. Therefore $\pi_L(a) \geq c \to a > a$ so $a \notin \pi L$. On the other hand, if $a \notin \pi L$ then $a$ has a successor $c$, $c \to a \in \max L$ by Lemma \ref{Lem:2}(3), and $(c \to a) \to a \geq c > a$.  
\end{proof}

Proposition \ref{Prop:13} can be reformulated to provide a first order condition for membership in $\pi L$.

\begin{proposition}\label{Prop:23}
	An element $a$ in a frame $L$ lies in $\pi L$ if and only $b \to a$ is not prime for any $b > a$. That is, $a \in \pi L$ if and only if
	\[
		\forall b > a\, \exists c_i\ \big( b \wedge c_1  \wedge c_2 \leq a \text{ but }(b \wedge c_1 \nleq a \text{ and } b\wedge c_2 \nleq a)\big).
	\]
\end{proposition}

\begin{notation*}[punctured element, unpunctured element]
	Let us agree to call an element $a$ of a frame $L$ \emph{punctured} if it has a successor in $L$, and \emph{unpunctured} if not.\footnote{This evocative terminology was suggested to the author by Fred Dashiell (\cite{Dashiell:2022}). } Thus an element $a$ is punctured if and only if it is of the form $c \wedge b$ for some $b > a$ and $c \in \max L$. Otherwise put, $a$ is punctured if and only if there exists some $b > a$ for which $b \to a \in \max L$.   
\end{notation*}

\subsection{An example: the pointless real numbers $\pi \Topol \mbb{R}$}

Consider the topology $\Topol\mbb{R}$ of the real numbers, whose skeleton $\delta \Topol \mbb{R}$ is the complete and atomless boolean algebra of regular open subsets of $\mbb{R}$. An open subset of the reals is uniquely expressible as a disjoint union of open intervals, and such a set is unpunctured if and only if no pair of its intervals are abutting, i.e., share an endpoint. A familiar example of an unpunctured open set is the complement of the Cantor set. The sublocale of unpuntured subsets of the reals forms a sublocale which evidently strictly contains the skeleton.

\begin{proposition}\label{Prop:16}
	$\pi \Topol \mbb{R}$ is the frame with generators $\ssetof{\overrightharp{p}, \overleftharp{q}}{\mbb{Q}}$ and the following relations indexed by the rational numbers $p,q \in \mbb{Q}$.
	\begin{enumerate}
		\item
		$\overrightharp{p} \vee \overleftharp{q} = \top$ if $p \leq q$.
		
		\item 
		$\overrightharp{p} \wedge \overleftharp{q} = \bot$ if $p > q$.
		
		\item 
		$\overrightharp{p} = \bigvee_{p < q} \overrightharp{q}$ and $\overleftharp{q} = \bigvee_{p < q}\overleftharp{p}$ for all $p$ and $q$.
		
		\item 
		$\bigwedge_\mbb{Q}\overrightharp{q} = \bot$ and $\bigvee_\mbb{Q} \overleftharp{p} = \top$.
	\end{enumerate}
\end{proposition}

The classical presentation of the locale of real numbers (\cite{PicadoPultr:2012}) differs from the presentation in Proposition \ref{Prop:16} in only one small detail: the $\leq$ operation the first relation is replaced by the $<$ operation in the presentation of the reals. However, any speculation that the (pointfree) Yosida adjunction (\cite{MaddenVermeer:1990}, \cite{BallHager:1990}) connecting $\mbf{F}$ with $\mbf{W}$ has a pointless counterpart founders on the observation that the family of pointfree real valued localic functions on a frame lacks (the frame counterparts of the) constant functions, and in particular lacks a $0$ function. (Here $\mbf{W}$ is the category of archimedean lattice-ordered groups with designated weak order unit.) Curiously, the family does have a negation operation. 

%$\overrightarrow{p}$

\subsection{Decomposing a frame into its scattered and atomless parts}

In classical topology a space is called \emph{scattered} if every nonempty closed subset contains an isolated point. We take the frame counterpart of this notion as our definition, while acknowledging that another definition is the starting point of the extensive and scholarly treatment of the topic in Chapter IX of the excellent reference \cite{PicadoPultr:2020}. See also \cite{BallPultr:2016}. 

\begin{definition*}[scattered frame, scattered element, coscattered element]
	A frame $L$ is called \emph{scattered} if every element $a < \top$ has a successor. An element $a$ of a frame $L$ is called \emph{scattered} if $\downset{a}_L$ (or $O_a$) is a scattered frame, and it is called \emph{coscattered} if $\upset{a}_L$ (or $C_a$) is a scattered frame. 
\end{definition*}

\begin{lemma}\label{Lem:7}
	The following hold in an arbitrary frame $L$ with $e \equiv \pi_L(\bot)$.
	\begin{enumerate}
		\item
		$e$ is the largest scattered element of $L$, and $\map{o_e}{L}{O_e}$ is the largest scattered open quotient of $L$.
		
		\item 
		$L$ is scattered if and only if $e = \top$ if and only if  $\bot$ is coscattered if and only if the pointless part of $L$ is empty.
		
		\item 
		$e = \bot$ if and only if the pointless part $\pi L$ is a dense quotient of $L$ if and only if $L$ is atomless.
	\end{enumerate}   
\end{lemma}

\begin{proof}
	The condition that $\pi L = \{\top\}$ is equivalent to the condition that $\pi_L^*(\bot) = \top$ by Proposition \ref{Prop:3}, which is to say that $\pi_L'(a) > a$ for all $a < \top$.
\end{proof}

\begin{definition*}[scattered and atomless parts of a frame]
	For a frame $L$, let $e \equiv \pi_L(\bot)$. We refer to the map $\map{o_e}{L}{O_e}$ or to its codomain as the \emph{scattered part of $L$}. We refer to the map $\map{c_e}{L}{C_e}$ or to its codomain as the \emph{atomless part of $L$}.
\end{definition*}

\begin{definition*}[subdirect product of frames]
	A subdirect product is a subobject of a product for which the projection morphisms are surjective.
\end{definition*}

\begin{proposition}\label{Prop:14}
	Every frame is (isomorphic to) a subdirect product of its scattered and atomless parts. Its scattered part is an open quotient and its atomless part is the complementary closed quotient. The frame is scattered if and only if its atomless part is empty and atomless if and only if its scattered part is empty.
\end{proposition}

\begin{proof}
	We have the product map $\map{o_e \times c_e}{L}{O_e \times C_e}$, which is one-one because the corresponding congruences $\Phi_e$ and $\Psi_e$ are complements in $\con L$ and therefore intersect to the identity congruence.	
\end{proof}

\subsection{Decomposing a frame into its pointless and spatial parts}\label{Subsec:Decomp}

\begin{lemma}\label{Lem:13}
	For any frame $L$, $\sigma_L \wedge \pi_L = \bot_{\con L} = 1_L$.
\end{lemma}

\begin{proof}
	We prove by induction that $\sigma_L(a) \wedge \pi_L^\alpha(a) = a$ for any $a \in L$ and any ordinal $\alpha$. For $\alpha = 0$ it is clear that $\sigma_L(a) \wedge \pi_L^0(a) = \sigma_L(a) \wedge a = a$. Assume the assertion holds for all ordinals $\beta < \alpha$. In case $\alpha = \beta + 1$ we have 
	\[
		\sigma_L(a) \wedge \pi_L^\alpha(a)
		= \sigma_L(a) \wedge \pi_L' \circ \pi_L^\beta (a) 
		= \sigma_L(a) \wedge \sbv{lr}{\pi_L^\beta(a)^+}c
		= \sbv{lr}{\pi_L^\beta(a)^+}(\sigma_L(a) \wedge c).
	\]
	If $c \in \pi_L^\beta(a)^+$ then either $c = \pi_L^\beta(a)$ or $c$ is a successor of $\pi_L^\beta(a)$. In the first case $\sigma_L(a) \wedge c = a$ by the inductive hypothesis. In the second case $b \equiv c \to \pi_L^\beta(a) \in \max L$ by Lemma \ref{Lem:2}(3), and since $b \geq \pi_L^\beta(a) \geq a$, it follows that $b \geq \bigwedge \upset{a}_{\max L} = \sigma_L(a)$. Therefore $\sigma_L(a) \wedge c \leq b \wedge c \leq \pi_L^\beta(a)$, so that by the inductive hypothesis we get $\sigma_L(a) \wedge c \leq \sigma_L(a) \wedge \pi_L^\beta(a) = a$. Since the persistence of the assertion through limit ordinals is evident, the induction is complete. 
\end{proof}

Whereas Lemma \ref{Lem:13} establishes that the congruences associated with the pointless and spatial parts of a frame are disjoint, Lemma \ref{Lem:35} points out that they are not complementary.

\begin{lemma}\label{Lem:35}
	In any frame $L$, $\sigma \pi L$ is empty, whereas $\pi \sigma L$ need not be empty. Otherwise put, $\sigma_L \circ \pi_L (a) = \top$ for all $a \in L$, whereas $\pi_L \circ \sigma_L(a)$ need not be $\top$ for all $a \in L$.
\end{lemma} 

\begin{proposition}\label{Prop:11}
	Every frame is (isomorphic to) a subdirect product of its pointless and spatial parts. The frame is pointless if and only if its spatial part is empty, and spatial if and only if its pointless part is empty.
\end{proposition}

\begin{proof}
	The subdirect representation is the product morphism of $\sigma_L$ with $\pi_L$. The product map is one-one by Lemma \ref{Lem:13}.
\end{proof}

\begin{proposition}\label{Prop:15}
	For any frame $L$ with $e \equiv \pi_L(\bot)$, we have the following commuting diagram in $\mbf{F}$. 
	\[
	\begin{tikzcd}
		&\pi L
		&&C_e \arrow{ll}
		&\\
		&&&&\\
		\pi L \times \sigma L \arrow{uur} \arrow{ddr}
		&& L \arrow{ll} \arrow{uul}{\pi_L} \arrow{ddl}[swap]{\sigma_L} \arrow{rr} \arrow{uur}[swap]{c_e} \arrow{ddr}{o_e}
		&& O_e \times C_e \arrow{uul} \arrow{ddl}\\
		&&&&\\
		&\sigma L \arrow{rr}
		&&O_e
		&
	\end{tikzcd}
	\]
	Concerning the two squares which have $L$ at one corner, the left square exhibits the decomposition of $L$ into its pointless and spatial parts (Proposition \ref{Prop:14}), while the right square exhibits the decomposition of $L$ into its scattered and atomless parts (Proposition \ref{Prop:11}). The top arrow exists because $C_e$ is the closure of $\pi L$, and the bottom arrow exists because $\sigma_L$ is the spatial reflector of $L$ and $O_e$ is spatial.   
\end{proposition}

\begin{corollary}\label{Cor:5}
	The pointless part of any frame $L$ is isomorphic to the pointless part of its atomless part. In symbols, $\pi L \approx \pi C_e$. 
\end{corollary}

\begin{proof}
	The elements of $\pi L$ are the elements of $L$ lacking successors in $L$, while the elements of $\pi C_e$ all the elements of $C_e$ lacking successors in $C_e$. Since  $\pi L \subseteq \upset{\pi_L(\bot)} = \upset{e} = C_e$, these two sets coincide. 
\end{proof}

We return to the topic of subdirect product decompositions of a frame in Subsection \ref{Subsec:EMdecomp}. 

\section{The category $\mbf{Fs}$}

In Section \ref{Sec:2Nuc} we took pains to emphasize the parallels between the spatial and pointless parts of a frame. But we must now acknowledge an important difference between the two: whereas the spatial part of a frame is reflective, the pointless part is not. However, by restricting the frame homomorphisms we can make the pointless part reflective as well.	

\subsection{Skinny frame homomorphisms}

\begin{lemma}\label{Lem:18}
	The following are equivalent for a frame homomorphism $\map{m}{L}{M}$.
	\begin{enumerate}
		\item 
		$m$ takes coscattered elements of $L$ to coscattered elements of $M$.
		
		\item 
		$m(\ker \pi_L) \subseteq \ker \pi_M$.
		
		\item 
		$m(\max L) \subseteq \ker \pi_M$, i.e., $\pi_M \circ m(a) = \top$ for all $a \in \max L$. 
		
		\item 
		$m$ drops through $\pi_L$ and $\pi_M$, i.e., there exists a unique map $\bar{m}$ such that $\pi_M \circ m = \bar{m} \circ \pi_L$.
		
		\item 
		$m$ takes unpunctured elements of $L$ to unpunctured elements of $M$.
	\end{enumerate}
\end{lemma}

\begin{proof}
	(2) certainly implies (3) because $\max L \subseteq \ker \pi_L$, and (3) implies (2) because $m^{-1}(\ker \pi_M)$ is a normal filter containing $\max L$ and $\ker \pi_L$ is the normal filter generated by $\max L$. The equivalence of (2) and (4) is an application of Lemma \ref{Lem:14}. 
\end{proof}

\begin{definition*}[skinny frame homomorphism, skinny contiinuous function]
	We refer to a frame homomorphism $\map{m}{L}{M}$ which satisfies the conditions of Lemma \ref{Lem:18} as being \emph{skinny}. The spatial counterpart of the notion of a skinny frame homomorphism is a continuous function which inversely preserves closed scattered subsets, or equivalently, a continuous function with scattered fibers. We shall refer to such functions as \emph{skinny functions}.   
\end{definition*}

\begin{definition*}[$\mbf{Fs}$, $\mbf{plFs}$]
	The category $\mbf{Fs}$ has objects which are frames and morphisms which are skinny frame homomorphisms. The full subcategory $\mbf{plFs}$ is comprised of the pointless frames.
\end{definition*}

Corollary \ref{Cor:1} establishes that $\mbf{Fs}$ is a legitimate category with $\mbf{plFs}$ as an  epireflective subcategory.

\begin{corollary}\label{Cor:1}
	\begin{enumerate}
		\item 
		A frame isomorphism is skinny.

		\item 
		A frame surjection is skinny.
		
		\item
		The composition of skinny frame homomorphisms is skinny.
		
		\item
		If $e$ is a frame surjection and $m$ is a frame homomorphism such that $m \circ e$ is skinny then $m$ is skinny.
		
		\item 
		$\mbf{plFs}$ is epireflective in $\mbf{Fs}$.
	\end{enumerate}
\end{corollary}

\begin{proof}
	(2) holds because Lemma \ref{Lem:3}(3) tells us that a frame surjection satisfies Lemma \ref{Lem:18}(3). (3) holds because the third property of Lemma \ref{Lem:18} is clearly preserved by composition. To check (4), consider frame homomorphisms $\map{e}{L}{K}$ and $\map{m}{K}{M}$ such that $m \circ e$ is skinny and $e$ is surjective. If $a$ is a maximal element of $K$ then $e_*(a) \equiv b$ is a maximal element of $L$ by Lemma \ref{Lem:14}(2), and $m(b) = a$. Since $m \circ e$ is skinny we get $\pi_M \circ m(a) = \pi_M \circ e(b) = \top$.  (5) is a consequence of Lemma \ref{Lem:18}(4).
\end{proof}

The diagram of Proposition \ref{Prop:15} exists in $\mbf{Fs}$.

\begin{proposition}
	All the mappings in the diagram of Proposition \ref{Prop:15} are skinny.
\end{proposition}

\begin{proof}
	The diagonal arrows are skinny by virtue of being surjective. The top and bottom arrows are skinny by  Corollary \ref{Cor:1}(4). The arrows $L \to \pi L \times \sigma L$ and $L \to O_e \times C_e$ can both be shown to induce bijections on maximal elements, and are therefore both skinny.
\end{proof}

\subsection{A factorization structure for $\mbf{Fs}$}

\begin{lemma}\label{Lem:12}
	The following are equivalent for a frame surjection $\map{e}{L}{M}$. The joins are taken in $\con L$.
	\begin{enumerate}
		\item 
		$\Theta_e 
		\leq \bigvee \setof{\Phi_a} {a \in e^{-1}(\top) \cap (\max L)}$. 
		
		\item 
		$\Theta_e 
		= \bigvee \setof{\Phi_a} {a \in e^{-1}(\top) \cap (\max L)}$.
		
		\item 
		$\Theta_e$ is a join of atoms.
	\end{enumerate}
\end{lemma}

\begin{proof}
	The equivalence of (1) and (2) is a consequence of \ref{Lem:8}, while the equivalence of (2) and (3) follows from Lemma \ref{Lem:10}.
\end{proof}

\begin{definition*}[$\bmfrak{E}$, $\bmfrak{M}$]
	Let $\bmfrak{E}$ be the class of $\mbf{Fs}$-surjections which satisfy the conditions of Lemma \ref{Lem:12}. Let $\bmfrak{M}$ be the class of all $\mbf{Fs}$-sources $(L \xrightarrow{m_i} M_i)_I$ such that for every $a \in \max L$ there exists an index $i \in I$ such that $m_i(a) < \top$. 
\end{definition*}

\begin{proposition}\label{Prop:1}
	$\mbf{Fs}$ has $\EM$-factorization of sources.
\end{proposition}

\begin{proof}
	Consider a source $\mcal{S} \equiv (L \xra{m_i} M_i)_I$ in $\mbf{Fs}$. Let 
	\[
		P
		\equiv \setof{a \in \max L}{\forall i \in I\ \big(m_i(a) = \top\big)},
	\]
	let $\Xi \equiv \bigvee_P \Phi_a$, and denote the quotient map by $\map{e}{L}{L/\Xi} \equiv \widehat{L}$. The map $e$ is skinny by Corollary \ref{Cor:1}(2). Note that $e(a) = \top$ for each $a \in P$ because $(a, \top) \in \Phi_a$ since $a \wedge a = a \wedge \top$ and $\Phi_a \subseteq \Xi$. It follows that $ e \in \bmfrak{E}$. Note that $\Phi_a \subseteq \Theta_{m_i}$ for any $a \in P$ and $i \in I$, for if $(a_1, a_2) \in \Phi_a$ then $a_1 \wedge a = a_2 \wedge a$, and since $m_i(a) = \top$ we get 
	\[	
		m_i(a_1) = m_i(a_1) \wedge m_i(a) = m_i(a_i \wedge a) = m_i(a_2 \wedge a) = m_i(a_2).		
	\]
	It follows that $\Xi \subseteq \Theta_{m_i}$, so that $m_i$ factors through $e$, say $m_i = \hat{m}_i \circ e$. The maps $\hat{m}_i$ are skinny by Corollary \ref{Cor:1}(4).
	
	Finally, we claim that the factored source $\widehat{\mcal{S}} \equiv (\widehat{L} \xra{\hat{m}_i} M_i)_I$ lies in $\bmfrak{M}$. For if $c$ is a maximal element of $\widehat{L}$ then $b \equiv e_*(c)$ is a maximal element of $L$ such that $(b, \top) \notin \Xi = \bigvee_P \Phi_a$. In particular, $b \notin P$ because otherwise $(b, \top) \in \Phi_b \subseteq \bigvee_P \Phi_a = \Xi$, contrary to assumption. It follows that there exists some index $i \in I$ such that $m_i(b) < \top$, which yields $\top > m_i (b) = \hat{m}_i \circ e(b) = \hat{m}_i(c)$. This proves the claim and the proposition. 
\end{proof}

\begin{proposition}\label{Prop:2}
	$\mbf{Fs}$ has the $\EM$-diagonalization property for sources.
\end{proposition}

\begin{proof}
	Consider a commuting square in $\mbf{Fs}$ with $e \in \bmfrak{E}$ and $(K \xra{m_i} M_i)_I \in \bmfrak{M}$.
	\[
	\begin{tikzcd}
		L \arrow{r}{e} \arrow{d}[swap]{f}
		&N \arrow{dl}[swap]{d} \arrow{d}{n_i}\\
		K \arrow{r}[swap]{m_i}
		& M_i
	\end{tikzcd}
	\]
	We claim that $P \equiv e^{-1}(\top) \cap (\max L) \subseteq f^{-1}(\top)$. For if $f(a) < \top$ for some $a \in e^{-1}(\top) \cap (\max L)$ then since $f(a) \in \ker \pi_K$ by Lemma \ref{Lem:18}(3), there must be an element $b \in \max K$ such that $b \geq f(a)$. Since $(K \xra{m_i} M_i)_I \in \bmfrak{M}$, there must be an index $i \in I$ for which $m_i(b) < \top$ in $M_i$, which leads to the contradiction
	\[
		\top
		> m_i(b) 
		\geq m_i \circ f(a)
		= n_i \circ e(a)
		= n_i(\top)
		= \top.
	\] 
	The claim shows that 
	\begin{align*}
		\Theta_e 
		&= \bigvee \setof{\Phi_a}{a \in e^{-1}(\top) \cap (\max L)} \\
		&\subseteq \bigvee \setof{\Phi_a}{a \in f^{-1}(\top) \cap (\max L)}
		\subseteq \Theta_f,
	\end{align*}
	with the result that the diagonal function $d$ exists by Lemma \ref{Lem:5}. This map is skinny by Corollary \ref{Cor:1}(4).
\end{proof}

\begin{proposition}\label{Prop:6}
	$\mbf{Fs}$ is an $\EM$-category.
\end{proposition}

\begin{proof}
	Propositions \ref{Prop:1} and \ref{Prop:2} show that $\mbf{Fs}$ has the features required by Definition 15.1 in \cite{AdamekHerrlichStrecker:2004}.
\end{proof}

\subsection{A second look at the pointless reflection in $\mbf{Fs}$}
The $\EM$-fact\-ori\-zation structure allows us to refine our understanding of  $\mbf{plFs}$ as an epireflective subcategory of $\mbf{Fs}$ (Corollary \ref{Cor:1}(5)). Theorem \ref{Thm:2} shows that it is actually an $\bmfrak{E}$-reflective subcategory, and this distinction will become important in Subsection \ref{Subsec:EMdecomp}.  

Pointlessness has several pleasing characterizations in terms of $\bmfrak{E}$-morphisms and $\bmfrak{M}$-morphisms.

\begin{lemma}\label{Lem:6}
	The following are equivalent for a frame $M$.
	\begin{enumerate}
		\item 
		Every frame homomorphism out of $M$ has a pointless codomain.
		
		\item 
		Every $\bmfrak{M}$-morphism into $M$ has a pointless domain.
				
		\item 
		Every $\bmfrak{E}$-morphism out of $M$ is an isomorphism.

		\item 
		$M$ is pointless.
	\end{enumerate} 
\end{lemma}

\begin{proof}
	If $M$ is not pointless then the identity morphism lies in both $\bmfrak{E}$ and $\bmfrak{M}$ and has both domain and codomain $M$. It follows that each of (1) and (2) implies (4). So suppose that $M$ is pointless. Then (2) must be true, for if $\map{m}{L}{M}$ is an $\bmfrak{M}$-morphism then any element $a \in \max L$ would map to an element $m(a) < \top$ in $M$, and since $m$ is skinny, we would have $\pi_M \circ m(a) = \top$. This cannot be the case, since the pointless nucleus $\pi_M$ acts as the identity on $M$. And (1) must also be true, for if $\map{m}{M}{L}$ is a frame homomorphism then any element $a \in \max L$ would produce the element $m_*(a) \in \max M$, contrary to assumption.
	
	If (3) holds then $M$ is isomorphic to $\pi M$ and is therefore pointless. And if $M$ is pointless and $\map{n}{M}{N}$ is an $\bmfrak{E}$-morphism then $n$, which is surjective by definition, must also be one-one because $\Theta_n \leq \bigvee_{\max M}\Phi_b = \bot_{\con M} = 1_M$. 
\end{proof}

\begin{lemma}
	$\mbf{plFs}$ is closed under the formation of $\bmfrak{M}$-sources.
\end{lemma}

\begin{proof}
	Consider a source $\mcal{S} \equiv (L \xra{m_i}M_i)_I$ in $\mbf{Fs}$ such that all the $M_i$'s are pointless. If $a$ is a maximal element of $L$ then because $\mcal{S} \in \bmfrak{M}$ there must be an index $i \in I$ for which $m_i(a) < \top$. Since $m_i$ is skinny, there must be a maximal $c \in L$ which lies above $m(a)$, contrary to our assumption that $M_i$ is pointless. We conclude that no maximal element can exist in $L$. 
\end{proof}

\begin{theorem}\label{Thm:2}
	$\mbf{plFs}$ is $\bmfrak{E}$-reflective in $\mbf{Fs}$.
\end{theorem}

\begin{proof}
	This is an instance of \cite[16.1]{AdamekHerrlichStrecker:2004}.
\end{proof}

\begin{proposition}\label{Prop:8}
	An $\mbf{Fs}$-homomorphism $\map{m}{L}{M}$ is isomorphic to $\pi_L$ if and only if $m \in \bmfrak{E}$ and $M$ is pointless.
\end{proposition}

\begin{proof}
	If $M$ is pointless then $m$ factors through the pointless reflector $\pi_L$, say $m = k \circ \pi_L$.
	\[
	\begin{tikzcd}
		L \arrow{r}{m} \arrow{d}[swap]{\pi_L}
		& M \arrow[shift right=.7]{dl}[swap]{l}\\
		\pi L \arrow[shift right=.7]{ur}[swap]{k}
		&
	\end{tikzcd}
	\]
	On the other hand, we claim that $\pi_L$ factors through $m$ by Lemma \ref{Lem:5}. For if $m(a) = \top$ for some $a \in L$ then since $m \in \bmfrak{E}$ we have 
	\begin{align*}
		(a, \top) &\in \bigvee \setof{\Phi_c}{c \in \max L \text{ and  }m(c) = \top}
		\subseteq \sbv{lr}{\max L} \Phi_c\\
		&= \setof{(a_1, a_2)}{\pi_L(a_1) = \pi_L(a_2)}.
	\end{align*}
	(The final equality is a consequence of Proposition \ref{Prop:4}.) We conclude that $\pi_L(a) = \pi_L(\top) = \top$, so that Lemma \ref{Lem:5} provides a homomorphism $\map{l}{M}{\pi L}$ such that $l \circ m = \pi_L$. Since both $m$ and $\pi_L$ are surjective, both $k$ and $l$ are unique with respect to their properties and are therefore inverse isomorphisms. 
\end{proof}

\begin{corollary}
	The top arrow in the diagram of Proposition \ref{Prop:15} is the pointless reflector of $C_e$.
\end{corollary}

\begin{proof}
	We leave to the reader the routine verification that the arrow $C_e \to \pi L$ is an $\bmfrak{E}$-morphism.
\end{proof}

\subsection{A second look at the spatial reflection}\label{Subsec:sFs}

The $\EM$-factorization structure also allows us to refine our understanding of the $\mbf{sF}$-reflection in $\mbf{F}$, for it manifests in the form of an $\bmfrak{M}$-reflection in $\mbf{Fs}$.

\begin{definition*}[$\mbf{sFs}$]
	We denote by $\mbf{sFs}$ the full subcategory of $\mbf{Fs}$ comprised of the spatial frames. 
\end{definition*}
 
 \begin{proposition}\label{Prop:10}
 	The category $\mbf{sFs}$ is $\bmfrak{M}$-reflective in $\mbf{Fs}$. In fact, the $\mbf{sF}$-reflector $\map{\sigma_L}{L}{\sigma L}$ of a frame $L$ is an $\bmfrak{M}$-surjection which functions as its $\mbf{sFs}$-reflector. 
 \end{proposition}
 
 \begin{proof}
 	Because the nucleus $\sigma_L$ fixes maximal elements it is clearly skinny and an $\bmfrak{M}$-morphism. To check the reflective property, consider a skinny frame homomorphism $m$ with spatial codomain $M$, and let $m = l \circ \sigma_L$ be its factorization given by Proposition \ref{Prop:4}. Then $l$ must be skinny by Corollary \ref{Cor:1}(4). 
 \end{proof}
 
 \begin{proposition}\label{Prop:5}
 	A frame homomorphism $\map{n}{L}{N}$ is isomorphic to $\sigma_L$ if and only if $n$ is an $\bmfrak{M}$-surjection and $N$ is spatial.  
 \end{proposition}
 
 \begin{proof}
 	If $N$ is spatial then it factors through the spatial reflector $\sigma_L$, say $n = k \circ \sigma_L$.
 	\[
 	\begin{tikzcd}
 		L \arrow{r}{n} \arrow{d}[swap]{\sigma_L}
 		&N\\
 		\sigma L \arrow{ur}[swap]{k}
 	\end{tikzcd}
 	\]
 	The map $k$ is surjective because $n$ is; to demonstrate that it is also one-one it is sufficient to establish the claim that the only element $a \in \sigma L$ for which $k(a) = \top$ is $a = \top$. For that purpose consider an element $a \in L$ such that $k(a) = \top$, and suppose for the sake of argument that $k(a) < \top$. Because $\sigma L$ is spatial there must be an element $b \in \max \sigma L$ such that $b \geq a$, hence $k(b) = \top$. But then $c \equiv \sigma_{L*}(b) \in \max L$ has the feature that $n(c) < \top$ because $n \in \bmfrak{M}$. In view of the fact that 
 	\[
	 	\top 
	 	> n(c) 
	 	= k \circ \sigma_L (c) 
	 	= k \circ \sigma_l \circ \sigma_{L*}(b)
	 	= k(b) = \top,
	 \] 
 	we have the contradiction which proves the claim.   
 \end{proof}
 
\section{The pointless and spatial reflections together}

In this section we address the interactions between the pointless and spatial parts of a frame. The two parts are bound together by a function which will play an important role in what follows. 

\subsection{The ligature binding the pointless and spatial parts of a frame}\label{Subsec:Lig}

\begin{definition*}[ligature $\map{\lambda_L}{\pi L}{\pi \sigma L}$]
	It is an immediate consequence of Theorem \ref{Thm:2} that for any frame $L$ there is a unique frame surjection $\lambda_L$ such that $\lambda_L \circ \pi_L = \pi_{\sigma L} \circ \sigma_L$. 
	\[
	\begin{tikzcd}
		L \arrow{r}{\sigma_L} \arrow{d}[swap]{\pi_L}
		& \sigma L \arrow{d}{\pi_{\sigma L}}\\
		\pi L \arrow{r}[swap]{\lambda_L}
		& \pi \sigma L
	\end{tikzcd}
	\]
	We refer to this map as the \emph{pointless ligature of $L$.}
\end{definition*}

%\begin{definition*}[spatial ligature of $L$, $\map{\gamma_L}{\max L}{\regfltr \pi L}$, $y_a$, spatial support $W_L$, ligature of $L$]
%	The points of a frame $L$ are bound to its pointless part by the function
%	\[
%		\map{\gamma_L}{\max L}{\regfltr \pi L}
%		= \Big(a \mapsto \setof{\pi_L(b)}{b \nleq a} =  \pi_L(x_a) \equiv y_a\Big).
%	\]
%	An important feature of this function is that its range is independent, i.e., distinct elements $a_i \in \max L$ give rise to independent regular filters $\pi(x_{a_i})$. (In spatial terms, a Tychonoff space is Hausdorff.) We refer to this function as the \emph{spatial ligature of $L$}, we refer to its range as the \emph{spatial support of L}, and we denote it by  
%	\[
%		W_L
%		\equiv \setof{y_a}{a \in \max L}.
%	\] 
%	The \emph{ligature} of $L$ is the pair $(\lambda_L, \gamma_L)$.
%\end{definition*}

Within a given frame $L$ reside the three sublocales $\pi L$, $\sigma L$, and $\pi \sigma L$. Lemma  \ref{Lem:20} provides a visualization of the actions of the three corresponding nuclei on an arbitrary frame element. This picture brings to light isomorphisms between certain subintervals of the frame.  
	
\begin{lemma}\label{Lem:20}
	For an element $a$ of a frame $L$, let $b 
	\equiv \lambda_L \circ \pi_L(a) = \pi_{\sigma L} \circ \sigma_L(a)$.
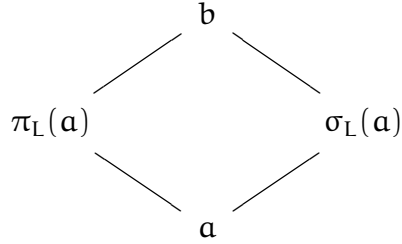
	\[
	\begin{tikzcd}
		&b
		&\\
		\pi_L(a) \arrow[dash]{dr} \arrow[dash]{ur}
		&& \sigma_L(a) \arrow[dash]{ul} \arrow[dash]{dl}\\
		&a
		&
	\end{tikzcd}
	\]
	Then $[a, \pi_L(a)]_L \approx [\sigma_L(a), b]_L$ and $[a, \sigma_L(a)]_L \approx [\pi_L(a), b]_L$.
\end{lemma}
	
\begin{proof}
	In view of the fact that $\pi_L(a) \wedge \sigma_L(a) = a$ by Lemma \ref{Lem:13}, this is an application of Lemma \ref{Lem:1}(1).
\end{proof}

%{\color{red}
%\todo{\tiny This Lemma is incorrect.  Can it be fixed?}
%Lemma \ref{Lem:21} shows that the pattern of Lemma \ref{Lem:20} can arise in only one way.
%
%\begin{lemma}\label{Lem:21}
%	In an arbitrary frame $L$, let $c$ be an element of $\pi L$ and $d$ be an element of $\sigma L$ such that $\lambda_L(c) = \pi_{\sigma L}(d)$. Then $c = \pi_L(a)$ and $d = \sigma_L(a)$ for $a = c \wedge d$. 
%\end{lemma}
%
%\begin{proof}
%	Clearly $\pi_L(a) \leq c$ because $\pi_L(a) = \bigwedge \upset{a}_{\pi L}$ and $a \leq c \in \pi L$, and $\sigma_L(a) \leq d$ likewise, so that $b \leq \lambda_L(c) = \pi_{\sigma L}(d)$, where $b \equiv \lambda_L \circ \pi_L(a) = \pi_{\sigma L} \circ \sigma_L(a)$ as in Lemma \ref{Lem:20}.  We have this diagram.
%	\[
%	\begin{tikzcd}
%		&\lambda_L(c) = \pi_{\sigma L}(d) \ar[dash]{dl} \ar[dash]{dr} \ar[dash]{d}
%		&\\
%		c \ar[dash]{d}
%		&b \ar[dash]{dl} \ar[dash]{dr} 
%		&d \ar[dash]{d}\\
%		\pi_L(a) \ar[dash]{dr}
%		&&\sigma_L(a) \ar[dash]{dl}\\
%		&a
%		&
%	\end{tikzcd}
%	\]
%	
%	Now $\sigma_L(a) = \sigma_L(c \wedge d) = \sigma_L(c) \wedge \sigma_L(d) = \sigma_L(c) \wedge d$, hence
%	\begin{align*}
%		b
%		&= \pi_{\sigma L} \circ \sigma_L(a)
%		= \pi_{\sigma L} \circ \sigma_L(c) \wedge \pi_{\sigma L}(d)
%		= \lambda_L \circ \pi_L(c) \wedge \pi_{\sigma L}(d)\\
%		&= \lambda_L (c) \wedge \pi_{\sigma L}(d)
%		= \lambda_L (c) 
%		= \pi_{\sigma L}(d)
%	\end{align*} 
%	Thus according to Lemma \ref{Lem:1}(1) the entire upper level of the diagram collapses, with the result that $\pi_L(a) = c$ and $\sigma_L(a) = d$.
%\end{proof}
%}

\subsection{The subdirect $\EM$-product decomposition of a frame}\label{Subsec:EMdecomp}
	 
\begin{definition*}[subdirect $\EM$-product]
	A \emph{subdirect $\EM$-product} of frames $E$ and $M$ is a subframe $L$ of the product frame $E \times M$ such that the projection map $L \to E$ is an $\bmfrak{E}$-surjection and the projection map $L \to M$ is an $\bmfrak{M}$-surjection. If $E$ happens to be pointless then the projection $L \to E$ is (isomorphic to) the pointless reflection of $L$ by Proposition \ref{Prop:8}. If $M$ happens to be spatial then the projection $\map{m}{L}{M}$ is (isomorphic to) the spatial reflection of $L$ by Proposition \ref{Prop:5}.
\end{definition*}

Our first task is to show that a subdirect $\EM$-product of a pointless frame $E$ with an arbitrary frame $M$ has a normal form. We produce this normal form by first identifying a subset $L'$ of the completely regular frame $E \times M$ which is closed under the frame operations, and then passing to its completely regular coreflection, i.e., to the largest completely regular frame $L$ contained in $L'$. This coreflection is a subdirect $\EM$-product of $E$ and $M$ as long as the projections are surjective. The latter condition is not generally met, but is met in the cases of interest (Theorem \ref{Thm:1}).   

\begin{proposition}\label{Prop:7}
	Let $E$ be a pointless frame and $M$ be a spatial frame, and let $\map{l}{E}{\pi M}$ be a frame surjection. Then  
	\[
		L'
		\equiv \setof{(a,b) \in E \times M}{l(a) = \pi_M(b)}
	\]
	is a naked subframe of $E \times M$, and if the completely regular coreflection $L \subseteq L'$ has surjective projections then 
	\begin{enumerate}
		\item
		$\max L = \setof{(\top, a)}{a \in \max M}$,
		
		\item 
		the projection $\map{m}{L}{M}$ is an $\bmfrak{M}$-morphism,
		
		\item 
		the projection $\map{e}{L}{E}$ is an $\bmfrak{E}$-morphism, and
		
		\item 
		$l \circ e = \pi_M \circ m$.
	\end{enumerate}
	In short, $L$ is a subdirect $(\bmfrak{E}, \bmfrak{M})$-product of $E$ and $M$ with ligature (isomorphic to) $l$. 
\end{proposition}

\begin{proof}
	Let us show that $L'$ is closed under the frame operations of $E \times M$. We can see that $L'$ contains both $\bot_{E \times M}$ and $\top_{E \times M}$. To show that $L'$ is closed under binary meets, simply note that if $(a_i, b_i) \in L'$ then 
	\[
		l(a_1 \wedge a_2) 
		= l(a_1) \wedge l(a_2) 
		= \pi_M(b_1) \wedge \pi_M(b_2) 
		= \pi_M(b_1 \wedge b_2).
	\] 
	To show that $L'$ is closed under arbitrary joins, consider a subset $\ssetof{(a_i, b_i)}{I} \subseteq L'$. Then $l\left(\bigvee_I a_i\right) = \bigvee_I l(a_i) = \bigvee_I \pi_L(b_i) = \pi_L \left(\bigvee_i b_i\right)$. 
	
	Now suppose that the projections $e$ and $m$ of the completely regular coreflection $L \subseteq L'$ are surjective; we get that $l \circ e = \pi_M \circ m$ by construction. 
	\[
		\begin{tikzcd}
			L \arrow{r}{m} \arrow{d}[swap]{e}
			&M \arrow{d}{\pi_M}\\
			E \arrow{r}[swap]{l}
			& \pi M
		\end{tikzcd}
	\]
	
	(1) It is obvious that the maximal elements of $L'$ are of the form $(\top, a)$, $a \in \max M$. Now consider a maximal element $(a,b)$ of $L$. We claim that $a = \top$, for if not then since $\max E = \emptyset$ there exists an element $a' \in E$ such that $\top > a' > a$, and since $e$ is surjective there exists an element $c \in M$ for which $(a', c) \in L$. Because $(a', c) \nleq (a, b) \in \max L$, it follows that $(\top, \top) = (a',c) \vee (a, b) = (a', b \vee c)$, hence $a' = \top$, a contradiction which proves the claim. Thus our maximal element of $L$ has the form $(\top, b)$.
	
	Our second claim is that $b \in \max M$. For $M$ is spatial and $b < \top$, so there exists an element $c \in \max M$ such that $b \leq c$. But since $m$ is surjective there must be an element $a \in E$ for which $(a, c) \in L$, hence $(a, c) \vee (\top, b) = (\top, c) \in L$, so $b = c$ by maximality. Finally, for every $c \in \max M$ the surjectivity of $m$ guarantees that $m_*(c)$ exists in $\max L$, and if $m_*(c) = (a, b)$ then the preceding argument establishes that $a = \top$ and $b = c$. 
	
	The bijective correspondence between $\max L$ and $\max M$ makes (2) obvious. To verify that $e$ is an $\bmfrak{E}$-morphism, note that for $(a_i, b_i) \in L$,
	\begin{gather*}
		e(a_1, b_1) = e(a_2, b_2)
		\iff a_1 = a_2 \implies \\
		\pi_M (b_1) = l(a_1) = l(a_2) = \pi_M(b_2),
	\end{gather*}
	so that with the aid of Lemma \ref{Lem:9} we get   
	\begin{gather*}
		\Theta_e
		\leq \setof{((a, b_1), (a, b_2))}{a \in E \text{ and }(b_1, b_2)  \in \Theta_{\pi_M}}\\
		= m^{-1}\left(\Theta_{\pi_M} \right)
		\leq m^{-1}\left(\sbv{r}{\max M}\Phi_c\right)
		= \sbv{lr}{\max M}m^{-1}(\Phi_c)\\
		= \sbv{lr}{\max M}\left(\Theta_m \vee \Phi_{m_*(c)} \right) 
		= \Theta_m \vee \sbv{lr}{\max M}\Phi_{m_*(c)}. 
	\end{gather*}
	In view of the facts that 
	\[
		\Theta_e \wedge \Theta_m 
		= \bot_{\con L} = 1_L
		\qtq{and} \max L 
		= \ssetof{m_*(c)}{\max M}, 
	\]
	we conclude that $\Theta_e \leq \bigvee_{\max M}\Phi_{m_*(c)} = \bigvee_{\max L} \Phi_a$. 
\end{proof}

\begin{notation*}[$E \times_l S$]
	For a pointless frame $E$, a spatial frame $M$, a frame surjection $\map{l}{E}{\pi M}$, we denote the frame $L$ of Proposition \ref{Prop:7} by 
	\[
		E \times_l M.
	\] 
\end{notation*}

\begin{theorem}\label{Thm:1} 
	Every frame is a subdirect $\EM$-product of its pointless and spatial parts. In detail, for any frame $L$ the product morphism $\map{\pi_L \times \sigma_L}{L}{\pi L \times \sigma  L}$ factors through the inclusion $\pi L \times_{\lambda L} \sigma L \to \pi L \times \sigma L$, and the initial factor is one-one.
	\[
	\begin{tikzcd}
		L \arrow{rr}{\sigma_L} \arrow{dr}{\tau_L} \arrow{dd}[swap]{\pi_L}
		&& \sigma L \\
		& \pi L \times_{\lambda_L} \sigma L \arrow{dr} \equiv \tau L
		&\\
		\pi L 
		&& \pi L \times \sigma L \arrow{ll} \arrow{uu}
	\end{tikzcd}
	\]
	We denote this initial factor by $\map{\tau_L}{L}{\pi L \times_{\lambda_L} \sigma L} \equiv \tau L$.
\end{theorem}

\begin{proof}
	The factorization is a consequence of the fact that $\lambda_L \circ \pi_L = \pi_{\sigma L} \circ \sigma_L$. The one-oneness of the product map is a consequence of Lemma \ref{Lem:13}, from which the one-oneness of its initial factor $\tau_L$ follows. The image $\tau_L(L)$ is contained in the subset $L'$ of Proposition \ref{Prop:7}, and since that image is completely regular, it is contained in the largest completely regular subset of $L'$, namely $\pi L \times_{\lambda_L} \sigma L$. 
\end{proof}

%{\color{red}
%Playing around.	
%\[
%\begin{tikzcd}[labels=description]
%	L \arrow["\sigma_L"]{r} \arrow["\pi_L"]{d}
%	&\sigma L \arrow["\pi_{\sigma L}"] {d}\\
%	\pi L \arrow["\lambda_L"]{r}
%	&\pi \sigma L
%\end{tikzcd}
%\]
%
%}

\subsection{The fat reflection in $\mbf{Fs}$}

\begin{definition*}[fat frames]
	Let us agree to call a frame \emph{fat} if the map $\map{\tau_L}{L}{\tau L}$ of Theorem \ref{Thm:1} is surjective. We denote the full subcategory of $\mbf{Fs}$ comprised of the fat frames by $\mbf{fFs}$.
\end{definition*}

\begin{theorem}\label{Thm:4}
	$\mbf{fFs}$ is bireflective in $\mbf{Fs}$, and a reflector for the frame $L$ is the map $\map{\tau_L}{L}{\tau L}$ of Theorem \ref{Thm:1}.
\end{theorem}

\begin{proof}
	Provided it is skinny, a test morphism $\map{m}{L}{M}$ factors through both the pointless and spatial reflectors of $L$ and $M$, and therefore engenders unique maps from each object of the diagram of Theorem \ref{Thm:1} for L to the corresponding object of the diagram for $M$. Since $M$ is fat, it is isomorphic to $\tau M$, so we end up with a unique arrow $\tau L \to M$ making both diagrams, and the arrows between them, commute.  
\end{proof}

We characterize the fat reflector of an atomless frame in Theorem \ref{Thm:3}.

\subsection{The fat question}

Theorem \ref{Thm:1} provides a method for investigating the structure of a frame by analyzing the interplay between its spatial and pointless parts, and this article can be regarded as first steps towards such an analysis. Although our results are preliminary rather than conclusive, we believe that the questions they raise are important. Chief among them is Question \ref{Ques:1}.

\begin{question}\label{Ques:1}
	What are the fat frames? Under what circumstances is the representation $\map{\tau_L}{L}{\tau L}$ surjective, i.e., when is $\tau_L(L) = \pi L \times_{\lambda_L} \sigma L$? Proposition \ref{Prop:19} shows an atomless frame with finite spatial part is always fat, but the general question is wide open. 
\end{question}

\section{Attaching points to a pointless frame}\label{Sec:AttPts}

In this section we construct various frames $L$ having a given pointless frame $E$ as their pointless part.  The task has already been done if $E$ is empty, since the scattered frames, and only the scattered frames, have the empty frame as their pointless part. So we assume throughout this section that $E$ is a given frame which is nonempty and pointless, and therefore also atomless.

\subsection{Attaching finitely many points to a pointless frame}

In this subsection we sprinkle finitely many points into the pointless frame $E$ at whim. That is, we show that it is possible to locate the points anywhere in $E$ we please, subject only to the necessity of using an independent family of regular filters to do so. 

\begin{notation*}[$\mbb{2}^I$, $\chix{J}$]
	We denote the two-element frame by $\mbb{2}$, and for an index set $I$ we denote the product of $I$ copies of $\mbb{2}$, i.e., the frame of functions $I \to \mbb{2}$, by $\mbb{2}^I$. Of course, every element of $\mbb{2}^I$ is the characteristic function of a unique subset  $J \subseteq I$; we denote this function by $\chix{J}$, the function itself being defined by the rule 
	\[
	\chix{J}(i)
	= \begin{cases*}
		\top     &\text{if $i \in J$}\\
		\bot     &\text{if $i \notin J$}
	\end{cases*}, \qquad i \in I.
	\]
\end{notation*}

\begin{proposition}\label{Prop:18}
	Let $W$ be a nonempty finite independent family of regular filters on $E$. For each $a \in E$, let $W_a \equiv \setof{w \in W}{a \in w}$. Then   
	\[
		L_W
		\equiv \setof{\splt{Y}{a} \in E \times \mbb{2}^W}{Y \subseteq W_a}
	\]
	is a completely regular subframe of $E \times \mbb{2}^W$ with the following features.
	\begin{enumerate}
		\item 
		$L_W$ is atomless.
		
		\item 
		$\max L_W = \ssetof{\splt{W \setminus \{w\}}{\top}}{W}$.
		
		\item 
		The spatial reflection of $L_W$ is  
		\begin{align*}
			\map{\sigma_{L_W}}{L_W}{\sigma L_W}
			&= \setof{\splt{Y}{\top}}{Y \subseteq W}\\
			&= \big(\splt{Y}{a} \mapsto \splt{Y}{\top}\big), 
			\quad a \in E,\,Y \subseteq W.
		\end{align*}
		The reflection frame $\sigma L_W$ is isomorphic to $\mbb{2}^W$, and the reflector is (isomorphic to) the projection onto the second factor:.

		\item 
		The pointless reflector of $L_W$ is
		\begin{align*}
			\map{\pi_{L_W}}{L_W}{\pi L_W}
			&= \ssetof{\splt{W_a}{a}}{E}\\
			&= \big(\splt{Y}{a} \mapsto \splt{W_a}{a}\big),
			\quad a \in E,\, Y \subseteq W.
		\end{align*}
		The reflection frame $\pi L_W$ is isomorphic to $E$, and the reflector is (isomorphic to) the projection onto the first factor.
		
		\item 
		The functions $\pi_{\sigma L_W}$ and $\lambda_{L_W}$ have empty codomains.
				
		\item 
		In terms of Proposition \ref{Prop:7}, $L_W$ is (isomorphic to) $E \times_{\lambda_{L_W}} \mbb{2}^W$.
	\end{enumerate}
\end{proposition}

\begin{proof}
	To show that $L_W$ is a subframe of $E \times \mbb{2}^W$, consider $\splt{Y_i}{a_i} \in L_W$. Then $	\splt{Y_1}{a_1} \wedge \splt{Y_2}{a_2} 	= \splt{Y_1 \cap Y_2}{a_1 \wedge a_2}$ and $a_1 \wedge a_2 \in \cap(Y_1 \cap Y_2)$ hence $\splt{Y_1}{a_1} \wedge \splt{Y_2}{a_2} \in L_W$. Likewise if $\ssetof{\splt{Y_i}{a_i}}{I} \subseteq L_W$ and $Y \equiv \bigcup_I Y_i$ then $	\bigvee_I \splt{Y_i}{a_i} = \splt{Y}{\bigvee_I a_i}$ and $\bigvee_I a_i \in \cap Y$, hence $\bigvee_I \splt{Y_i}{a_i} \in L_W$. The proof that $L_W$ is completely regular is completed by the Lemmas \ref{Lem:26}, \ref{Lem:27}, and \ref{Lem:28}, in which the symbols retain their meaning in the theorem statement. 
	
	(1) Suppose for the sake of argument that $\splt{Y}{a}$ is an atom of $L_W$. If $Y = \emptyset$ then $a$ would have to be an atom of $E$, which is ruled out by hypothesis. If $Y$ contains two distinct elements $y_i \in Y$ then we would have $\bot < \splt{Y\setminus \{y_1\}}{a} \neq \splt{Y\setminus \{y_2\}}{a} < \splt{Y}{a}$, contrary to assumption. So $Y$ must be a singleton, say $Y = \{y\}$, and $a$ must lie in $y$. But Lemma \ref{Lem:24} assures the existence of an element $b \in y$ such that $b < a$, so that we have the contradiction $\bot < \splt{Y}{b} < \splt{Y}{a}$. We conclude no atom exists in $L_W$.
	
	(2) Since $E$ is pointless, it is clear that the only predecessors of $\top_{E \times \mbb{2}^W} = \splt{W}{\top}$ are of the form $\splt{W\setminus\{w\}}{\top}$ for elements $w \in W$. (3) Thus the maximal elements above a typical member $\splt{Y}{a} \in L_W$ are those of the form $\splt{W\setminus\{w\}}{\top}$ for $w \in W\setminus Y$, and their meet is $\splt{Y}{\top}$. (4) And with Lemma \ref{Lem:2}(3) in mind, it is clear that the successors of a typical element $\splt{Y}{a}$ are of the form $\splt{Y \cup\{w\}}{a}$ for elements $w \in W \setminus Y$, and the join of these successors is then $\splt{W_a}{a}$. 	
\end{proof} 

\begin{notation*}[partition $W = Y \oplus Z$]
	We write $W = Y \oplus Z$ to indicate that subsets $Y, Z \subseteq W$ partition $W$, i.e., $Y \cup Z = W$ and $Y \cap Z = \emptyset$.	
\end{notation*}

\begin{lemma}\label{Lem:26}
	For any nontrivial partition $W = Y \oplus Z$ and for any $z \in Z$,  
	\[
		z' \equiv \setof{b \in z}{b\st \in \cap Y} \neq \emptyset
		\qtq{and} \bigvee_{z'} b\st = \top.
	\]
\end{lemma}

\begin{proof}
	Since $z$ is regular, it is enough to show that $z'$ is nonempty. But since $W$ is independent, for each $y \in Y$ there exists an element $b_y \in z$ such that $b_y\st \in y$, yielding $\bigwedge_Y b_y \in z'$. 
\end{proof}

\begin{lemma}\label{Lem:27}
	Let $W = Y \oplus Z$ be a nontrivial partition, let $\map{t}{Z}{E}$ be a function such that $t(z) \in z'$ for all $z \in Z$, and let $a,c \in \cap Y$ be such that $c \combel a \leq \bigwedge_Z t(z)\st$. Then both $\splt{Y}{c}$ and $\splt{Y}{a}$ lie in $L_W$, and $\spltt{Y}{c} \combel \splt{Y}{a}$.  
\end{lemma}	

\begin{proof}
	It is sufficient to show that under the weaker assumption $a \prec c$ we can conclude that $\splt{Y}{c} \prec \splt{Y}{a}$. For in the presence of this weaker result, the stronger hypothesis that $c \combel a$ posits the existence of a witnessing family $\ssetof{a_q}{\mbb{Q}}$ such that $c \leq a_q \leq a$ for all $q \in Q$, whereupon the weaker conclusion produces a witnessing family $\ssetof{\splt{Y}{a_q}}{\mbb{Q}}$ establishing the stronger conclusion $\splt{Y}{c} \combel \splt{Y}{a}$.
	
	Therefore assume that $c \prec a$, and consider the element $\splt{Z}{c\st \vee \bigvee_Z t(z)}$, an element which lies in $L_W$ by construction. We claim that this element witnesses $\splt{Y}{c} \combel \splt{Y}{a}$. The claim relies on the fact that $\big(\bigvee_Z t(z)\big)^* = \bigwedge_Z t(z)\st$, so that $c \wedge \bigvee_Z t(z) = \bot$. Therefore
	\begin{gather*}
		\spltt{Y}{c} \wedge \spltt{Z}{c^* \vee \bigvee\nolimits_Z t(z)} = 
		\splt{\emptyset}{\bot_E}
		= \bot_{L_W}
		\qtq{and}\\
		 \splt{Y}{a} \vee \spltt{Z}{c^* \vee \bigvee\nolimits_Z t(z)} 
		 = \splt{W}{\top_E}
		 = \top_{L_W}. \qedhere 
	\end{gather*}
\end{proof}

\begin{lemma}\label{Lem:28}
	$L_W$ is completely regular.
\end{lemma}

\begin{proof}
	Given the element $\splt{Y}{a} \in L_W$, put $Z \equiv W \setminus Y$. Let $T$ be the family of all functions $\map{t}{Z}{E}$ such that $t(z) \in z'$ for all $z \in Z$. For each $t \in T$ let 
	\[
		C(a, t)
		\equiv \setof{c \in \cap Y}{c \combel a \wedge \bigwedge\nolimits_Z t(z)\st}.
	\]
	Note that since $a \in \cap Y$ and each filter $y \in Y$ is round, $\bigvee C(a, t) = a \wedge \bigwedge_Z t(z)\st$.  Then 
	\begin{gather*}
		\bigvee_T\sbv{r}{C(a, t)} \spltt{Y}{c}
		= \spltt{Y}{\bigvee_T\bigvee{C(a, t)}} =
		\spltt{Y}{\bigvee_T \left(a \wedge \bigwedge_Z t(z)\st \right)} =\\	
		\spltt{Y}{a \wedge \bigvee_T \bigwedge_Z t(z)^*}
		= \spltt{Y}{a \wedge \bigwedge_Z \bigvee_T t(z)\st}
		= \spltt{Y}{a}. \qedhere
	\end{gather*}
\end{proof}

Lemma \ref{Lem:28} completes the proof of Proposition \ref{Prop:18}. Proposition \ref{Prop:19} then completes the circle of ideas by showing that $L_W$ is characterized by its properties.  This  requires the consideration of the spatial support of a frame.

\subsection{The spatial support of a frame}

\begin{notation*}[$x_a$, $y_a$, spatial support]
	In a frame $L$ with pointless part $E$, we denote the ideals in $L$ and $E$ generated by an element $a \in \max L$ by
	\[
		x_a
		\equiv \setof{b \in L}{b \nleq a}
		\qtq{and}
		y_a
		\equiv \setof{b \in E}{b \nleq a}.
	\] 
	The \emph{spatial support of $L$ is $\ssetof{y_a}{\max L}$}.
\end{notation*}

\begin{lemma}
	In a frame $L$ with pointless part $E$ and maximal element $a$,
	\[
		y_a 
		= \pi_L(x_a), \quad 
		x_a 
		= \pi_L^{-1}(y_a), \qtq{and}
		y_a = x_a \cap E.
	\]
\end{lemma}

\begin{lemma}\label{Lem:25}
	Let $L$ be an atomless frame with pointless part $E$. 
	\begin{enumerate}
		\item 
		If $x$ is a proper regular filter on $L$ then $\pi_L(x)$ is a proper regular filter on $E$.
		
		\item 
		If $x$ is a maximal proper round filter on $L$ then $\pi_L(x)$ is a proper regular filter on $E$.
		
		\item 
		If $A$ is a subset of $\max L$ then $\ssetof{\pi_L(x_a)}{A}$ is an independent family of proper regular filters on $E$.
	\end{enumerate}
\end{lemma}

\begin{proof}
	(1) We can identify $E$ with $\pi L$, and think of $\pi_L$ as a dense surjection from $L$ onto $E$. One consequence is that $\pi_{L*}(\bot) = \bot$, so that $\pi_L(x)$ is proper because $x$ is proper. Another consequence is that the skeleton map $\delta_L$ factors through $\pi_L$ so that the skeleton itself is contained in $\pi L$. That implies that for an element $a \in \pi L$, its  pseudocomplements in $L$ and $\pi L$ coincide. Thus for a regular filter $x$ on $L$,  
	\[
	\top = \bigvee_x b\st 
	\implies \top = \pi_L\left(\bigvee_x b\right)
	= \bigvee_x \pi_L(b\st)
	= \bigvee_x \pi_L(x)\st. 
	\]
	
	(2) If $x$ is a maximal proper round filter on $L$ which is not regular then it is of the form $x_a$ for some $a \in \max L$ by Lemma \ref{Lem:22}(5), in which case $\pi_L(x)$ is regular by Lemma \ref{Lem:23}(4).
	
	(3) If $a_1$ and $a_2$ are distinct maximal elements of $L$ then $x_{a_1}$ and $x_{a_2}$ are distinct maximal proper round filters by Lemma \ref{Lem:22}(6), and therefore contain disjoint elements $b_i \nleq a_i$ by Lemma \ref{Lem:22}(4). Since $\pi_L$ is a dense surjection, $\pi_L(b_1)$ and $\pi_L(b_2)$ are disjoint elements of 
	$\pi_L(x_{a_1})$ and $\pi_L(x_{a_2})$, respectively.
\end{proof}

\begin{lemma}\label{Lem:24}
	A proper regular filter on any frame cannot contain a least element. 
\end{lemma}

\begin{proof}
	If a proper regular filter $x$ on a frame $L$ contained a least element $a$ then $a$ would have to be completely below itself and therefore be complemented. But in that case $\bigvee_x b\st = a\st < \top$, contrary to hypothesis. 
\end{proof}

\begin{proposition}\label{Prop:19}
	An atomless frame $M$ having pointless part $E$ and finite spatial support $W$ is isomorphic to $L_{W}$. Explicitly, the map 
	\[
		\map{k}{M}{L_W}
		\equiv \left(a \mapsto \spltt{W_a}{\pi_M(a)}\right), \qquad  a \in M,
	\] 
	is an $\bmfrak{E}$-isomorphism, and in light of of Proposition \ref{Prop:18}(6), is (isomorphic to) the map $\tau_M$ from Theorem \ref{Thm:1}. In particular, an atomless frame with finite spatial part is fat.  
\end{proposition}

\begin{proof}
	The reader will have no difficulty verifying that $k$ preserves binary meets and therefore preserves and reflects order. The proof that $k$ is a frame isomorphism is completed by showing that $k$ is bijective. That $k$ is one-one follows from Lemma \ref{Lem:13}, for if $k(a_1) = k(a_2)$ then $\pi_M(a_1) = \pi_M(a_2)$, and $\sigma_M(a_1) = \sigma_M(a_2)$ because $a_1$ and $a_2$ lie below the same maximal elements by Proposition \ref{Prop:18}(2).
	
	To check that $k$ is surjective, consider an element $\splt{Z}{b} \in L_W$. By definition of $L_W$, $b \in E$ and $b \in \bigcap_Z y_a$. Let 
	\[
		A 
		\equiv \setof{a \in \max M}{b \in y_a \notin  Z},
	\]
	and put $b' \equiv b \wedge \bigwedge A$. We aim to show $k(b') = \splt{Z}{b}$. Surely $\pi_M(b') = \pi_M(b) \wedge \bigwedge_A \pi_M(a) = \pi_M(b) = b$; what remains to be shown is that $Z = \setof{y_a}{b' \in y_a}$. But this is clear, for by construction it is the case that for any $a \in \max L$, $b' \leq a$ if and only if $b \leq a$ or $a \in A$. That is, $b' \in y_a$ if and only if $b \in y_a$ and $a \notin A$, i.e., if and only if $y_a \in Z$. Finally, $k$ is an $\bmfrak{E}$-morphism because for $b_i \in M$, $k(b_1) = k(b_2)$ if and only if $\pi_M(b_1) = \pi_M(b_2)$ .
\end{proof}

Proposition \ref{Prop:17} is preparation for Subsection \ref{Subsec:InfManyPts}

\begin{proposition}\label{Prop:17}
	Let $M$ be an atomless frame with pointless part $E$, let $X$ be a nonempty finite subset of its spatial support $W$, and let $\Xi \equiv \bigvee \setof{\Phi_a}{y_a \notin X}$. Then the quotient map $\map{k_X}{M}{M/\Xi}$ is an $\bmfrak{E}$-morphism with codomain (isomorphic to) $L_X$.  
\end{proposition}

\begin{proof}
	The map $k_X$ is an $\bmfrak{E}$-morphism by construction, and its adjoint $k_{X*}$ provides a bijection from $\max (M/\Xi)$ onto $X$. The result follows from Proposition \ref{Prop:19}.   
\end{proof}

\subsection{Attaching infinitely many points to a pointless frame}\label{Subsec:InfManyPts}

The next task is to analyze atomless frames with infinite spatial support.

\begin{lemma}\label{Lem:30}
	Let $X$ and $Y$ be nonempty finite independent families of regular filters on $E$, let $\map{r}{Y}{X}$ be a function such that $y \supseteq r(y)$ for all $y \in Y$, and let 
	\[
		\map{r^{-1}}{\mbb{2}^X} {\mbb{2^Y}} 
		= \big(\chix{Z} \mapsto \chix{r^{-1}(Z)}\big), \qquad Z  \subseteq X, 
	\]
	be the frame map induced by $r$. Then the map 
	\[
		\map{l^X_Y}{L_X}{L_Y}
		\equiv \big(\splt{Z}{a} \mapsto \splt{r^{-1}(Z)}{a}\big),
		\qquad a \in E,\, Z \subseteq X,
	\]
	is a skinny frame injection. 
\end{lemma}

\begin{proof}
	A maximal element of $L_X$, which has the form $\splt{X\setminus \{x\}}{\top}$ for some $x \in X$, maps to $\splt{Y\setminus r^{-1}(x)}{\top} \equiv b$. Enumerate $r^{-1}(x)$ as $\{y_i\}_{i \leq n}$. If $n = 0$ then $b = \splt{Y}{\top} = \top_{L_Y}$. If $n > 0$ then for each $i \leq n$, 
	\[
		b_i
		\equiv \splt{Y\setminus \setof{y_j}{j \neq i}}{\top}
	\]
	is a successor of $b$, and $\bigvee_{i \leq n} b_i = \top$. In either case $\pi_{L_Y}(b) = \top$, meaning that $l^X_Y$ is skinny by Lemma \ref{Lem:18}(3).  
\end{proof}

\begin{notation*}[$Y \subseteq_\omega W$]
	Let $W$ be a family of filters on a frame $L$. We write $Y \subseteq_\omega W$ to indicate that $Y$ is a finite subset of $W$.  
\end{notation*}

\begin{definition*}[plenary family of regular filters on $E$]
	We shall call a family $W$ of regular filters on $E$ \emph{plenary} if it is nonempty and independent, and for all $\bot < a \in E$ there exists an element $w \in W$ such that $a \in w$, i.e., $W_a \neq \emptyset$. 
\end{definition*}

\begin{proposition}\label{Prop:20}
	Let $W$ be a plenary family of regular filters on $E$, and for subsets $Y \subseteq X \subseteq_\omega W$, let
	$\map{l^X_Y}{L_X}{L_Y}$ be the surjection of Lemma \ref{Lem:30} arising from the inclusion $Y \to X$. The $\mbf{nF}$-pullback (limit) 
	\[
		\setof{L_W' \xra{l^W_X} L_X}{X \subseteq_\omega W}
		\text{ of the diagram}
		\setof{L_X \xra{l^X_Y} L_Y}{Y \subseteq X \subseteq_\omega W}
	\]
	is (isomorphic to) $\setof{\splt{Z}{a} \in E \times \mbb{2}^W}{Z \subseteq W_a}$
	with projections 
	\[
		\map{l^W_X}{L_W}{L_X}
		= \left(\splt{Z}{a} \mapsto \splt{Z \cap X}{a}\right),
		\quad Z \subseteq W, \; a \in E. 
	\]
	Then the completely regular coreflection $L_W$ of $L_W'$ is the $\mbf{Fs}$-pullback of the diagram. 
	
	If the projections are surjective then $L_W$ is atomless and 
	\[
		\max L_W 
		= \max L_W' 
		= \big\{\,\splt{W \setminus \{w\}}{\top} : w \in W\,\big\}.
	\]
	If the projections are $\bmfrak{E}$-morphisms then the pointless reflector of $L_W$ is
		\begin{align*}
			\map{\pi_{L_W}}{L_W}{\pi L_W}
			&= \setof{\splt{W_a}{a}}{a \in E}\\
			&= \big(\splt{Z}{a} \mapsto \splt{W_a}{a}\big),
			\quad \splt{Z}{a} \in L_W.
		\end{align*}
	The reflection frame $\pi L_W$ is isomorphic to $E$, and the reflector is (isomorphic to) the projection onto the first factor.
\end{proposition}

\begin{proof}
	A moment's reflection is all that is necessary to check the first paragraph, so we begin by verifying that $L_W$ is atomless when the projections are surjective. Were it to exist, an atom of $L_W$ would have the form $\spltt{Z}{a}$ for $\bot_{L_W} = \splt{\emptyset}{\bot_E} < \splt{Z}{a}\in L_W'$. It follows that $a > \bot$ in $E$, for otherwise $Z \subseteq W_a = W_\bot = \emptyset$ so $\splt{Z}{a} = \splt{\emptyset}{\bot}$, contrary to assumption. Since $W$ is plenary, there exists an element $w \in W_a$, so that for any subset $X \subseteq_\omega W$ containing $w$, the surjection $l^W_X$ would take $\splt{Z}{a}$ to $\splt{Z \cap X}{a} \in L_X$. But since $L_X$ is atomless by Proposition \ref{Prop:18}(1), there exists an element $\splt{Y}{b} \in L_X$ such that $\bot_{L_X} < \splt{Y}{b} < \splt{Z \cap X}{a}$, and since $l^W_X$ is surjective, there exists an element $\splt{V}{b} \in L_W$ such that $l^W_X\splt{V}{b} = \splt{V \cap X}{b} = \splt{Y}{b}$. It follows that $\splt{V \cap Z}{b} = \splt{V}{b} \wedge \splt{Z}{a}$ lies in $L_W$ and satisfies $\bot_{L_W} < \splt{V \cap Z}{b} < \splt{Z}{a}$, contrary to assumption. We conclude that $L_W$ is atomless.
	
	A maximal element $\splt{Z}{a} \in L_W$ has the feature that each of its projections $l^W_X\splt{Z}{a} = \splt{Z \cap X}{a}$, $X \subseteq_\omega W$, must be either the top element of $L_X$ or a maximal element of $L_X$, and the second case must occur. To reiterate, for all $X \subseteq_\omega W$ either $\splt{Z \cap X}{a} = \splt{X}{\top}$ or $\splt{Z \cap X}{a} = \splt{X\setminus \{x\}}{\top}$ for some $x \in X$. In any case $a = \top$, and the element $x$ in the second case is unique. For if $x_1$ and $x_2$ are distinct elements of $X$ for which there exist subsets $X_i \subseteq_\omega W$ such that $x_i \in X_i$ and $l^W_{X_i}\splt{Z}{\top} = \splt{X_i\setminus\{x_i\}}{\top}$ then $X \equiv X_1 \cup X_2 \subseteq_\omega W$ and $l^W_X\splt{Z}{\top} = \splt{X\setminus\{x_1,x_2\}}{\top}$, a clear contradiction since $\splt{X\setminus\{x_1, x_2\}}{\top}$ is not maximal in $L_X$. In short, a maximal element of $L_W$ must be a maximal element of $L_W'$. Furthermore, the surjectivity of the projections forces every maximal element of $L_W'$ to appear as a maximal element of $L_W$.
	 		
	If the projections are $\bmfrak{E}$-morphisms then the $\mbf{Fs}$-pullback is a source comprised of $\bmfrak{E}$-morphisms. Likewise the sink $\setof{L_X \xra{\pi_{L_X}} E}{X \subseteq_\omega W}$ is comprised of $\bmfrak{E}$-mor\-phisms, and since the composition of $\bmfrak{E}$-morphisms is an $\bmfrak{E}$-morphism (\cite[5.1.8(2)]{AdamekHerrlichStrecker:2004}), the composition of this source with this sink is a single $\bmfrak{E}$-morphism $L_W \to E$. This morphism must be the pointless reflector of $L_W$ by Proposition \ref{Prop:8}. 
\end{proof}

\begin{notation*}[$L_W$ for arbitrary $W$]
	For a plenary family $W$ of regular filters on a pointless frame $E$, we denote the frame constructed in Proposition \ref{Prop:20} by $L_W$. 
\end{notation*}

We summarize.

\begin{theorem}\label{Thm:3}
	An atomless frame $M$ with pointless part $E$ and spatial support $W$ is isomorphic to a subframe of $L_W$ with surjective projections. In detail, the source $\setof{\map{k_X}{M}{L_X}}{X \subseteq_\omega W}$ provided by Proposition \ref{Prop:17} factors through the pullback $L_W$ of Proposition \ref{Prop:20}, and this embedding is the fat reflector of $M$ (Theorem \ref{Thm:4}).
\end{theorem}

Theorem \ref{Thm:3} focuses attention on the fundamental Question \ref{Ques:2}.

\begin{question}\label{Ques:2}
	In a given pointless frame $E$, which plenary families of regular filters serve as the spatial support of atomless frames with pointless part $E$?
\end{question} 

\section{The spatial support of a frame}

In this section we focus on the interaction between a frame and its spatial support. In Subsection \ref{Subsec:SpCom} we characterize the spatiality and compactness of the frame in terms of its spatial support, and in Subsection \ref{Subsec:MaxRndFltr} we look into the situation that arises when all of the filters in the spatial support of the frame are maximal proper round filters on its pointless part. 

\subsection{The spatial support determines spatiality and compactness}\label{Subsec:SpCom}
We show in Proposition \ref{Prop:22} that a necessary condition for an atomless frame to be spatial is that it must contain enough regular filters in its spatial support to distinguish the elements of its pointless part. This requires the preparatory Lemma \ref{Lem:33}, in connection with which recall the inductive definition of $\pi_L$ given in Section \ref{Sec:Prelim}.

\begin{lemma}\label{Lem:33}
	For an element $b$ of a frame $L$, define 
	\[
		A'(b)
		\equiv \setof{a \in \max L}{a \geq b \text{ and }a \to b >  b},
	\] 
	and for ordinal numbers $\beta$ define
	\begin{align*}
			A^\beta(b) &\equiv \{b\}&&\text{if $\beta = 0$,}\\
			A^{\beta}(b) &\equiv A^\alpha(b) \cup A'(\pi_L^\alpha(b))&&\text{if $\beta = \alpha + 1$,}\\
			A^\beta(b) &\equiv \bigcup_{\alpha < \beta} A^\alpha(b)&&\text{if $\beta$ is a limit ordinal, and} \\
			A(b) &\equiv A^\beta(b) &&\text{for some (any) $\beta$ such that $A^\beta(b) = A^{\beta + 1}(b)$.}
		\end{align*}	
	Then $b = \pi_L(b) \wedge \bigwedge A(b)$.
\end{lemma}

\begin{proof}
	We first claim that $b = \pi_L'(b) \wedge \bigwedge A'(b)$ for any $b \in L$. For 
	\begin{gather*}
		\left(\bigwedge A'(b)\right) \wedge \pi'_L(b)  
		= \left(\bigwedge A'(b)\right) \wedge \left(b \vee \sbv{lr}{A'(b)}(a \to b)\right)\\
		= \left(b \wedge \bigwedge A'(b)\right) \vee \sbv{lr}{A'(b)}\left(\left(\bigwedge A'(b)\right) \wedge (a \to b)\right)\\
		= b \vee \sbv{lr}{A'(b)}\left(\left(\bigwedge A'(b)\right) \wedge (a \to b)\right) = b.
		\end{gather*}
	We then use induction to show that $b = \pi_L^\beta(b) \wedge \bigwedge A^\beta(b)$ for all $\beta$. The assertion clearly holds for $\beta = 0$, so assume it holds for all ordinals $\alpha < \beta$. If $\beta = \alpha + 1$ then
	\begin{gather*}
			\pi_L^\beta(b) \wedge \bigwedge A^\beta(b)
			= \pi_L^{\alpha + 1}(b) \wedge \bigwedge A^{\alpha + 1}(b) =\\
			\pi_L' \circ \pi_L^\alpha(b) \wedge \bigwedge\left(A^\alpha(b) \cup A'(\pi_L^\alpha(b))\right)
			= \\
			\left(\pi_L' \circ \pi_L^\alpha(b) \wedge \bigwedge A'(\pi_L^\alpha(b))\right) \wedge \bigwedge A^\alpha(b) =\\
			\pi_L^\alpha(b) \wedge \bigwedge A^\alpha(b) = b.
		\end{gather*}
	(The penultimate equality holds by the claim and the ultimate equality by the induction hypothesis.) If $\beta$ is a limit ordinal then we get
	\begin{gather*}
			b \leq \pi_L^\beta(b) \wedge \bigwedge A^\beta(b)
			= \sbv{lr}{\alpha < \beta} \pi_L^\alpha(b) \wedge \bigwedge \sbcup{lr}{\gamma < \beta} A^\gamma(b) =\\
			\sbv{lr}{\alpha < \beta}\left(\pi_L^\alpha(b) \wedge \bigwedge\sbcup{lr}{\gamma < \beta}A^\gamma(b)\right)
			\leq \sbv{lr}{\alpha < \beta}\left(\pi_L^\alpha(b) \wedge \bigwedge A^\alpha(b)\right)
			= b. \qedhere
		\end{gather*}
\end{proof}

\begin{proposition}\label{Prop:22}
	The following are equivalent for a frame $L$ with pointless part $E$. 
	\begin{align}
			&\forall b \in L\ \left(\upset{b}_L = \sbcap{lr}{b \nleq a \in \max L}x_a = \sbcap{lr}{b \in x_a}x_a \right),\\
			&\forall b \in L\ \left( b = \sbw{lr}{b \leq a \in \max L}a \ \ \right),\qquad\text{i.e., $L$ is spatial,}\\
			&\forall b \in E\ \left( b = \sbw{lr}{b \leq a \in \max L}a \ \ \right).		
		\end{align}
	When these conditions obtain then 
	\[
		\forall b \in E\ \left(\upset{b}_E = \sbcap{lr}{b \nleq a \in \max L} y_a \   \right).
	\]
\end{proposition}

\begin{proof}
	(2) is equivalent to (3), for if (3) holds then (2) follows from Lemma \ref{Lem:33} because $A(b) \cup \upset{\pi_L(b)}_{\max L} \subseteq \upset{b}_{\max L}$. 	Assume (1) and suppose for the sake of argument that (2) fails at $b \in L$, i.e., $b < b' \equiv \bigwedge \upset{b}_{\max L}$. Then $b \notin \upset{b'}_L$, so that by (1) there must be some element $a \in \max L$ for which $b \notin x_a \ni b'$, i.e., $b' \nleq a \geq b$, contrary to assumption. 	Conversely, assume (2) to prove (1). Clearly we have $\upset{b}_L \subseteq \bigcap \setof{x_a}{b \in x_a, a \in \max L}$ for any $b \in L$. If $b \nleq b' \in L$ then by (2) there exists a maximal element $a$ such that $b \nleq a \geq b'$, i.e., $b' \notin x_a \ni b$.
	
	If $L$ satisfies the numbered conditions then for any $b \in E$ we have
	\begin{align*}
			\upset{b}_E
			&= E \cap \upset{b}_L
			= E \cap \bigcap \setof{x_a}{b \in x_a,\, a \in \max L}\\
			&= \bigcap \setof{E \cap x_a}{b \in x_a, \, a \in \max L}\\
			&= \bigcap \setof{y_a}{b \in y_a, \, a \in \max L} \qedhere 
		\end{align*}		    
\end{proof}

\begin{question}
	Are the numbered conditions of Proposition \ref{Prop:22} equivalent to the condition given in the last sentence?
\end{question}

\begin{proposition}\label{Prop:21}
	Let $L$ be an atomless spatial frame with pointless part $E$ and spatial support $W$. Then $L$ is compact if and only if $W$ is maximal among independent families of regular filters on $E$.
\end{proposition}

\begin{proof}
	If $W$ is not maximal then there exists a regular filter $z$ on $E$ such that $W \cup \{z\}$ is independent, i.e., for each $w \in W$ there exist an element $b \in z$ such that $b\st \in w$. Then $\bigvee_z b\st = \top_L$, for otherwise the spatiality of $L$ would imply the existence of some $a \in \max L$ for which $b\st \leq a$ for all $b \in z$, i.e., $b\st \notin y_a$ for all $b \in z$, contrary to assumption. But no finite subset $z' \subseteq_\omega z$ has the feature that $\bigvee_{z'}b\st = \top$, for that would imply that the element $b_0 \equiv \bigwedge z' \in z$ had the feature that $b_0\st = \top$, an impossibility. The point here is that $L$ is not compact. 
	
	Now suppose that $W$ is maximal, and suppose for the sake of argument that $C$ is a cover of $L$ without a finite subcover. By replacing $C$ with the the ideal it generates, we may assume that $C$ is a proper ideal. By replacing $C$ with $\setof{b}{\exists c \in C\ (b \combel c)}$, we may assume that $C$ is proper and round. $C$ is therefore contained in a maximal proper round ideal $C'$, and $x \equiv \setof{c\st}{c \in C'}$ is a maximal proper round filter on $L$ by Lemma \ref{Lem:22}(10), so that $y \equiv \pi_L(x)$ is a regular filter on $E$ by Lemma \ref{Lem:25}. 
	
	By virtue of its maximality, $W$ must contain at least one filter $w$ with which $y$ has the finite intersection property, which is to say that $a \wedge b > \bot$ for each $a \in w$ and $b \in y$. Since $W$ is the spatial support of $L$, there exists a unique element $a \in \max L$ for which $w = y_a = \setof{b \in E}{b \nleq a}$. We claim that $x = x' \equiv \setof{b \in L}{b \nleq a}$. For if $x$ differs from $x'$ then $x$ and $x'$ would contain disjoint elements by Lemma \ref{Lem:22}(4), an eventuality which is ruled out by the observation that $\pi_L(x') = w$ and $\pi_L(x) = y$, and $w$ and $y$ have the finite intersection property. The important point here is that $a = \bigvee_x b\st$ by Lemma \ref{Lem:22}(5).
	
	The argument is concluded by noticing that for each $c \in C$ we have $c\st \in x$, hence $c\stst \leq \bigvee_x b\st = a$, from which follows $\bigvee C \leq a$. This contradicts our assumption that $C$ covered $L$. We conclude that $L$ is compact.   
\end{proof}

\begin{corollary}
	An atomless frame $L$ whose spatial support is maximal among independent families of round filters has the feature that $\upset{b}_L = \bigwedge_{b \in x_a} x_a$ for all $b \in L$. 
\end{corollary}

\begin{proof}
	A compact frame is spatial.
\end{proof}

\section{When the spatial support consists of maximal proper round filters}\label{Sec:MaxRndFltr}

In this section we show that the condition of complete regularity, which intrudes into our fundamental Proposition \ref{Prop:20}, is automatically satisfied when the spatial support consists of maximal proper round filters. The result is Theorem \ref{Thm:5}, which requires some preparation. 

\subsection{Attaching points at maximal proper round filters}

\begin{lemma}\label{Lem:29}
	Let $W$ be a plenary family of maximal proper round filters on $E$. Then   
	\[
		L_W'
		\equiv \setof{\spltt{Z}{a} \in E \times \mbb{2}^W}{Z  \subseteq W_a}
	\]
	is a naked subframe of $E \times \mbb{2}^W$ with the following features.
	\begin{enumerate}
		\item 
		$L_W'$ is atomless.
		
		\item 
		$\max L_W' = \ssetof{\spltt{W \setminus \{w\}}{\top}}{W}$.
		
		\item 
		$L_W'$ contains $E' \equiv \ssetof{\spltt{W_a}{a}}{E}$ as a subset isomorphic to $E$ in the inherited order, and the map $E \to E' = \big(a \mapsto \spltt{W_a}{a}\big)$ preserves $\bot$, $\top$, and binary meets. The map also preserves the completely below relation in the following sense.
		
		\item 
		If $\ssetof{b_p}{\mbb{Q}}$ is a witnessing family for $a_1 \combel a_2$ in $E$ then $\ssetof{\spltt{W_{b_p}}{b_p}}{\mbb{Q}}$ is a witnessing family for $\spltt{W_{a_1}}{a_1} \combel \spltt{W_{a_2}}{a_2}$ in $E'$ and also in $L_W'$. In particular, $a_1 \combel a_2$ in $E$ implies $\spltt{W_{a_1}}{a_1} \combel \spltt{W_{a_2}}{a_2}$ in $L_W'$.		
	\end{enumerate}
\end{lemma}

\begin{proof}
	To show that $L_W'$ is closed under the frame operations in $E \times \mbb{2}^W$, first note that $L_W'$ contains both $\bot_{E \times \mbb{2}^W} = \splt{\emptyset}{\bot}$ and $\top_{E \times \mbb{2}^W} = \splt{W}{\top}$. If $\spltt{Z_i}{a_i} \in L_W'$ then 
	\begin{gather*}
		Z_1 \cap Z_2 \subseteq W_{a_1} \cap W_{a_2} = W_{a_1 \wedge a_2}, \text{ hence }\\
		\spltt{Z_1}{a_1} \wedge \spltt{Z_2}{a_2} = \spltt{Z_1 \cap Z_2}{a_1 \wedge a_2} \in L_W'.
	\end{gather*} 
	Likewise if $\left\{\spltt{Z_i}{a_i}\right\}_I \subseteq L_W'$ and $Z \equiv \bigcup_I Z_i$ then $\bigvee_I a_i \in \cap Z$, hence $\bigvee_I \spltt{Z_i}{a_i} = \spltt{Z}{\bigvee_I a_i} \in L_W'$. 
	
	(1) Suppose for the sake of argument that $\spltt{Z}{a}$ is an atom of $L_W'$. If $Z = \emptyset$ then $a$ would have to be an atom of $E$, which is ruled out by hypothesis. If $Z$ contains two distinct elements $z_i \in Z$ then we would have 
	\[
		\bot 
		< \spltt{Z\setminus \{z_1\}}{a} 
		\neq \spltt{Z\setminus \{z_2\}}{a} 
		\leq \spltt{Z}{a},
	\] 
	contrary to assumption. So $Z$ must be a singleton, say $Z = \{z\}$, and $a$ must lie in $z$. But Lemma \ref{Lem:24} assures the existence of an element $b \in z$ such that $b < a$, so that we have the contradiction $\bot < \spltt{Z}{b} < \spltt{Z}{a}$.
	
	(2) Since $E$ is pointless, it is clear that the only predecessors of $\top_{E \times \mbb{2}^W} = \spltt{W}{\top}$ are of the form $\spltt{W\setminus\{w\}}{\top}$ for elements $w \in W$. 
	
	(4) It is sufficient to establish the claim that $a_1 \combel a_2$ in $E$ implies $\spltt{W_{a_1}}{a_1} \prec \spltt{W_{a_2}}{a_2}$ with witness in $L_W'$. So suppose $a_1 \combel a_2$ in $E$ with witnessing family $\left\{b_p\right\}_\mbb{Q}$, and fix rational numbers $p < q$.  Since $b_p\st \vee b_q = \top$, $b_q \combel a_2$ with witnessing family $\left\{b_r\right\}_{q < r}$, and $b_p\st \combel a_1\st$ with witnessing family $\left\{b_r\st\right\}_{r < p}$, it follows from Lemma \ref{Lem:22}(9) that $W_{a_2} \cup W_{a_1\st} = W$. Since $W_{a_1\st} \cap W_{a_1} = \emptyset$, we have shown that $\spltt{W_{a_1}}{a_1} \prec \spltt{W_{a_2}}{a_2}$ with witness $\spltt{W_{a_1\st}}{a_1\st}$.   
\end{proof} 

\begin{lemma}\label{Lem:32}
	Let $W$ and $L_W'$ be as in Lemma \ref{Lem:29}. Then for any $w \in W$, 
	\[
		\bigvee_w \spltt{W_{b\st}}{b\st}
		= \spltt{W\setminus\{w\}}{\top},
	\]
	the supremum being reckoned in $E \times \mbb{2}^W$.
\end{lemma}

\begin{proof}
	The supremum $\bigvee_w b\st$ is $\top_E$ because a maximal proper round filter on a pointless frame is regular by Lemma \ref{Lem:22}(5). The argument is completed by showing that  $\bigcup_w W_{b\st} = W \setminus \{w\}$. For $w \notin W_{b\st}$ for any $b \in y$ since $b\st \notin w$. And for any $x \in W$, $x \neq w$, there exists an element $b \in w$ for which $b\st \in x$ because $W$ is an independent family. The point is that any such $x$ lies in $\bigcup_w W_{b\st}$.
\end{proof}

\begin{lemma}\label{Lem:31}
	Let $W$, $L_W'$, and $E'$ be as in Lemma \ref{Lem:29}, and let 
	\[
		L_W''
		\equiv \setof{\spltt{Z}{a}}{W_a \setminus Z \subseteq_\omega W}.
	\] 
	Then $L_W''$ is the smallest subset of $L_W'$ which contains $E' \cup \big\{\splt{W\setminus\{w\}}{\top}\big\}_W$ and is closed under binary meets. Furthermore, each member of $L_W''$ is the join (in $E \times \mbb{2}^W$) of its lower bounds in $E'$.  
\end{lemma}

\begin{proof}
	To verify the first statement, observe that a typical element $\spltt{Z}{a} \in L_W''$ can be expressed as
	\[
		\spltt{W_a}{a} \wedge \sbw{lr}{W_a\setminus Z} \splt{W \setminus \{w\}}{\top}.
	\]
	To verify the second, use Lemma \ref{Lem:32} to expand this expression. 
	\begin{align*}
		\spltt{W_a}{a} \wedge \sbw{lr}{W_a\setminus Z}\splt{W\setminus\{w\}}{\top}
		= \spltt{W_a}{a} \wedge \sbw{l}{W_a\setminus Z}\bigvee_w \spltt{W_{b\st}}{b\st}\\
		= \bigvee \setof{\spltt{W_a}{a} \wedge  \spltt{W_{b_w\st}}{b_w\st}}{b_w\in w \in W_a\setminus Z}.\quad \qedhere 
	\end{align*}
\end{proof}

\begin{theorem}\label{Thm:5}
	Let $W$, $L_W'$, $L_W''$, and $E'$ be as in Lemma \ref{Lem:31}, and let $L_W^\star$ be the family of all joins (in $E \times \mbb{2}^W$) of elements of $L_W''$. Then $L_W^\star$ has the following properties.
	\begin{enumerate}
		\item 
		$L_W^\star$ is the  smallest (completely regular) subframe of $E \times \mbb{2}^W$ containing $E' \cup \big\{\splt{W\setminus\{w\}}{\top}\big\}_W$.
		
		\item 
		$\max L_W^\star = \big\{\splt{W\setminus\{w\}}{\top}\big\}_W$.
		
		\item 
		The pointless part of $L_W^\star$ has $E$ as a quotient, in the sense that $E'$ is a sublocale of $\pi L_W^\star$, and if we denote the quotient map by $\map{q}{\pi L_W^\star}{E'}$, then  
		\begin{align*}
			\map{q \circ \pi_{L_W^\star}}{L_W^\star}{E'}
			= \left(\spltt{Z}{a} \mapsto \spltt{W_a}{a}\right),
			\quad  a \in E,\: Z\subseteq W.
		\end{align*} 
	\end{enumerate}  
\end{theorem}

\begin{proof}
	(1) The fact that $L_W^\star$ is completely regular follows from Lemma \ref{Lem:29}(3) together with the second sentence of Lemma \ref{Lem:31}. 
	
	(4) To show that $E'$ is closed under arbitrary meets, consider a family $\ssetof{\spltt{W_{a_i}}{a_i}}{I} \subseteq E'$, and let $a_0 \equiv \bigwedge_I a_i$ in $E$. We claim that $\bigwedge_I \spltt{W_{a_i}}{a_i} = \spltt{W_{a_0}}{a_0}$ in $L_W^\star$. Since $W_{a_0} \subseteq W_{a_i}$ for all $i$, it is clear that $\spltt{W_{a_0}}{a_o} \leq \spltt{W_{a_i}}{a_i}$ for all $i$. But for any $\spltt{Z}{b} \in L_W^\star$ such that $\spltt{Z}{b} \leq \spltt{W_{a_i}}{a_i}$ for all $i$, it must be the case that $b \leq a_0$, and hence that $Z \subseteq W_b \subseteq W_{a_0}$.  
	
	To complete the proof that $E'$ is a sublocale of $L_W^\star$, we shall show that $\spltt{Z}{b} \to \spltt{W_a}{a} = \spltt{W_{b \to a}}{b \to a}$ for any $\spltt{Z}{b} \in L_W^\star$ and $\spltt{W_a}{a} \in E'$. Surely $\spltt{W_{b \to a}}{b \to a} \wedge \spltt{Z}{b} = \spltt{W_{b \to a} \cap Z}{(b \to a) \wedge b} \leq \spltt{W_a}{a}$, for any filter in $W_{b \to a} \cap Z$ contains both $b \to a$ and $b$ and hence contains $a$. And if $\spltt{Z'}{c} \wedge \spltt{Z}{b} = \spltt{Z' \cap Z}{c \wedge b} \leq \spltt{W_a}{a}$ then $c \leq b \to a$ hence $Z' \cap Z \subseteq W_{c \wedge b} \subseteq W_{c} \subseteq W_{b \to a}$.   
\end{proof}

\begin{proposition}\label{Prop:24}
	Let $W$ be the family of maximal proper round filters on $E$. Then the frame $L_W^\star$ of Theorem \ref{Thm:5} is (isomorphic to) $\beta E$, the compact coreflection (aka \v{C}ech-Stone compactification) of $E$. 
\end{proposition}

\begin{proof}
	Notice that $L_W^\star$ is compact by Proposition \ref{Prop:21}. According to the classical construction of Banaschewski and Mulvey (\cite{BBMulvey:1980}, \cite{BBMulvey:1984}), we may take $\beta E$ to be the frame of round ideals on $E'$.  Therefore it is enough to establish that the maps
	\begin{align*}
		\map{g}{L_W^\star}{\beta E} &= \Big(\spltt{Z}{b} \mapsto \setof{\spltt{W_a}{a}}{\spltt{W_a}{a} \combel \spltt{Z}{b}}\Big)\\
		\beta E \to L_W^\star &=  \left(u \mapsto \bigvee u\right)
	\end{align*}
	are inverse frame isomorphisms. Clearly $\bigvee g\splt{Z}{b} = \splt{Z}{b}$ for all $\splt{Z}{b} \in L_W^\star$, for $\splt{Z}{b}$ is the join of its lower bounds in $L_W''$ by construction, each element of $L_W''$ is the join of its lower bounds in $E'$ by Lemma \ref{Lem:31}, and by Lemma \ref{Lem:29}, each element of $E'$ is the join, in both $E'$ and $L_W^\star$, of those elements of $E'$ completely below it.  
	
	Consider an element $u \in \beta E'$ with $\bigvee u \equiv \splt{Z}{b}$. Because $u$ is round, each element of $u$ has an element of $u$ completely above it in $E'$ and therefore also in $L_W^\star$ by Lemma \ref{Lem:29}(4), from which it follows that $u \subseteq g\big(\bigvee u\big)$. The opposite containment is just as clear, for a round ideal of open subsets of a compact Hausdorff space is precisely the family of open sets completely contained in its union.  
\end{proof}

\begin{corollary}\label{Cor:6}
	For an atomless frame $L$ with pointless part $E$, the cardinality of $\max L$ is bounded above by the cardinality of $\max \beta E$. 
\end{corollary}

Example \ref{Ex:2} makes the point that the pointless part of $L_W^\star$ (Theorem \ref{Thm:5}(4)) may be strictly larger than $E$.

\begin{example}\label{Ex:2}
	Let $E$ be the regular open algebra of the topology $\Topol \mbb{R}$ of the real numbers, and let $W$ be the family of all maximal round filters on $E$. Then $L_W^\star$ is isomorphic to $\beta E$ by Proposition \ref{Prop:24}, and according to the proof of the proposition, we may take it to be the frame of round ideals on $E$. Since in this case every ideal is round, $L_W^\star$ is (isomorphic to) the topology of the Stone space dual to $E$, i.e., the topology of the Gleason cover of the real numbers. 
	
	The pointfree part of $L_W^\star$ is the sublocale consisting of those elements which satisfy the condition of Proposition \ref{Prop:23}. The ideals which satisfy this  condition include the principal ideals, but include other ideals as well. For instance, the reader will have no difficulty in verifying that the ideal of regular open subsets of finite measure is nonprincipal and satisfies the condition.     
\end{example}

\begin{question}\label{Ques:3}
	Is every pointless frame the pointless part of a compact frame?
\end{question}

Proposition \ref{Prop:24}, together with Example \ref{Ex:2}, motivate the consideration in Subsection \ref{Subsec:qPcorefl} of compact frames whose pointless parts are $C^*$-quotients. This requires a brief digression in Subsection \ref{Subsec:sFsdeTychs}. 

\subsection{$\mbf{sFs}$ is dually equivalent to $\mbf{Tychs}$}\label{Subsec:sFsdeTychs} 

Here we make the point in Proposition \ref{Prop:9} that the restriction to skinny frame homomorphisms, i.e., the passage from $\mbf{F}$ to $\mbf{Fs}$, does not invalidate the classical dual equivalence between spatial frames and Tychonoff spaces.  

\begin{definition*}[$\mbf{Tychs}$, {$\map{\max}{\mbf{sFs}}{\mbf{Tychs}^{\op}}$}]
	We denote the category of Tychonoff spaces with skinny continuous functions by $\mbf{Tychs}$. The functor 
	\[
		\map{\max}{\mbf{sFs}}{\mbf{Tychs}^{\op}} 
	\]
	associates with each spatial frame $L$ the space with carrier set $\max L$ topologized by declaring
	\[
		\mcal{O}\max L
		\equiv \ssetof{\upset{a}_{\max L}}{L},
	\] 
	and associates with each $\mbf{Fs}$-morphism $\map{m}{L}{M}$ the continuous function 
	\[
		\map{\max m}{\max M}{\max L}
		= \big(b \mapsto m_*(b)\big).
	\] 
\end{definition*}

Proposition \ref{Prop:9} shows that $\mbf{sFs}$ and $\mbf{Tychs}$ are dually equivalent categories.

\begin{proposition}\label{Prop:9}
	The functor $\map{\max}{\mbf{sFs}}{\mbf{Tychs}^{\op}}$ is an equivalence of categories.
\end{proposition}

\begin{proof}
	We have already remarked that a continuous function between Tychonoff spaces has a frame counterpart which is skinny if and only if its fibers are scattered, a correspondence which is evidently full and faithful. Of course, every spatial frame is the topology of a Tychonoff space. See \cite[3.33]{AdamekHerrlichStrecker:2004}.
\end{proof}

\subsection{The $qP$-coreflection in compact atomless frames}\label{Subsec:qPcorefl}

Proposition \ref{Prop:25} points out an exotic feature of compact Hausdorff spaces whose pointless parts are $C^*$-embedded. 

\begin{proposition}\label{Prop:25}
	Let $X$ be a compact Hausdorff space without isolated points whose pointless part is $C^*$-embedded, by which we mean that its topology $\mcal{O}X \equiv L$ is isomorphic to $\beta \pi L$. Then the removal of any single point does not change $\mcal{C}X$, in the sense that for any $x \in X$, $\mcal{C}X = \mcal{C}^*(X \setminus \{x\})$. 
\end{proposition}

\begin{proof}
	Denote the open quotient corresponding to $X \setminus \{x\}$ by $\map{o_x}{L}{O_x}$. Since $\pi_L$ factors through $o_x$ by Proposition \ref{Prop:4}, the latter is a $C^*$-quotient. 
\end{proof}

\begin{question}
	In terms of Proposition \ref{Prop:25}, if each quotient $\map{o_x}{L}{O_x}$, $x \in X$, is a $C^*$-quotient, does it follow that $\pi_L$ is a $C^*$-quotient? 
\end{question}

The spaces described in Proposition \ref{Prop:25} are reminiscent of $P$-spaces. A point of a Tychonoff space is called a \emph{$P$-point} if every real-valued function on the space is constant on a neighborhood of the point. A $P$-space is a space whose every point is a $P$-point. (See \cite{BallWWZenk:2011} and the references therein for a thorough discussion of $P$-spaces.) Obviously any non-isolated point of a $P$-space $X$ can be removed without changing $\mcal{C}^*X$.  But a compact $P$-space is finite, whereas we are about to demonstrate that the spaces of Proposition \ref{Prop:25} abound. In fact, they are reflective in compact Hausdorff spaces without isolated points, meaning every such space has a canonical embedding into a space with the properties of Proposition \ref{Prop:25}. 

\begin{definition*}[$qP$-frame, $qP$-space]
	A \emph{$qP$-frame} is a compact frame whose pointless part is a $C^*$-quotient. That is, a frame $L$ is a $qP$-frame if $\map{\pi_L}{L}{\pi L}$ is isomorphic to $\map{\beta_{\pi L}}{\beta \pi L}{\pi L}$.\footnote{The term \enquote*{quasi-$P$ frame} was used in \cite{DubeNsondeNayi:2015} with another meaning.} A $qP$ space is a Tychonoff space whose topology is a $qP$-frame. 
\end{definition*} 

\begin{lemma}\label{Lem:34}
	A compact frame with a pointless $C^*$-quotient is a $qP$-frame. Otherwise put, a $qP$-frame is a frame of the form $\beta M$ for some pointless frame $M$.
\end{lemma}

\begin{proof}
	If $\map{m}{L}{M}$ is a $C^*$-quotient map with $L$ compact and $M$ pointless then, because it must factor through $\pi_L$, it follows that $\pi_L$ is a $C^*$-quotient isomorphic to $\beta_{\pi L}$.  
\end{proof}

\begin{corollary}
	A $qP$-frame is atomless.
\end{corollary}

\begin{proof}
	Because a $qP$-frame $L$ is of the form $\beta M$ for a pointless frame $M$, the coreflector $\map{\beta_M}{\beta M}{M}$ is surjective and therefore takes atoms to atoms, and because $M$ is atomless by Lemma \ref{Lem:2}(2), it follows that $L$ is atomless.
\end{proof}

\begin{definition*}[$\mbf{kalFs}$, $\mbf{qPFs}$]
	We denote by $\mbf{kalFs}$ and by $\mbf{qPFs}$ the full subcategories of $\mbf{Fs}$ comprised, respectively, of the compact atomless frames and of the $qF$-frames.
\end{definition*}

\begin{theorem}\label{Thm:6}
	$\mbf{qPFs}$ is coreflective in $\mbf{kalFs}$; a coreflector for the frame $L$ is the map $\map{q_L}{\beta \pi L}{L}$ such that $\pi L \circ q_L = \beta_{\pi_L}$.
\end{theorem}

\begin{proof}
	Given a test $\mbf{kalFs}$-morphism $m$ with $M$ a $qP$-frame, apply the pointless reflection to get $\pi m$, to which apply the compact coreflection to get $\beta \pi m$, and then insert the maps $q_L$ and $q_M$ such that $\pi_L \circ q_L = \beta_{\pi L}$ and $\pi_M \circ q_M = \beta_{\pi M}$.  
	\[
	\begin{tikzcd}
		M \arrow{r}{m} \arrow{d}{\pi_M} 
		&L \arrow{d}[swap]{\pi_L} \\
		\pi M \arrow{r}{\pi m}
		&\pi L\\
		\beta \pi M \arrow{u}[swap]{\beta_{\pi M}} \arrow{r}[swap]{\beta \pi m} \arrow[bend left]{uu}{q_M}
		&\beta \pi L \arrow{u}{\beta_{\pi L}} \arrow[bend right]{uu}[swap]{q_L}
	\end{tikzcd}
	\]
	Then $q_M$ is an isomorphism because $\pi_M$ is a $C^*$-quotient map and is therefore isomorphic to $\beta_{\pi M}$ (cf.\ Lemma \ref{Lem:34}), so that $m = q_L \circ (\beta \pi m) \circ q_M^{-1}$ is the desired factorization of $m$.
\end{proof}

\begin{corollary}\label{Cor:4}
	The full subcategory of $\mbf{Tychs}$ comprised of the $qP$-spaces is bireflective in the full subcategory of $\mbf{Tychs}$ comprised of the compact spaces without isolated points.
\end{corollary}

The author would like to express his gratitude to the CECAT gang for stimulating and encouraging discussions regarding the topics under investigation. These Wednesday afternoon Zoom sessions were organized by Chapman University's Center of Excellence in Computation, Algebra, and Topology. In particular, the author would like to thank Andrew Moshier for suggesting the proof of Lemma \ref{Lem:4}.

%\end{multicols}

	\end{document}